\documentclass[11pt]{article}

\usepackage[font=small,format=plain,labelfont=bf,up]{caption}

\usepackage[a4paper,nohead,ignorefoot,%
twoside,scale=0.85,bindingoffset=1.25cm]{geometry}

\usepackage{amssymb,amsfonts,amsmath,amsthm}

\usepackage[]{graphicx}
\usepackage{color}

\graphicspath{{figures/}}

\newcommand{\url}[1]{#1} % dummy definition of \texorpdfstring
\usepackage{hyperref}%
\definecolor{gray}{rgb}{0.2,0.2,.2}
\hypersetup{%
  unicode=true,          % non-Latin characters in Acrobat’s bookmarks
  colorlinks=true,       % false: boxed links; true: colored links
  linkcolor=gray,          % color of internal links
  citecolor=gray      % color of links to bibliography
}

\definecolor{colorBlue}{rgb}{0.,.0,0.75}
\definecolor{colorGreen}{rgb}{0.,0.65,.0}
\definecolor{colorRed}{rgb}{0.99,0.,.0}

\newcommand{\BMHC}{}
\newcommand{\EMHC}{}

\newcommand{\BMHR}{}
\newcommand{\EMHR}{}

\newcommand{\iu}{\mathtt{i}}

\newcommand{\fspace}[1]{{\mathsf{#1}}}
\newcommand{\fspaceL}{\fspace{L}}

\newcommand{\ol}[1]{{\overline{#1}}}

\newcommand{\Rset}{{\mathbb{R}}}

\newcommand{\Nset}{{\mathbb{N}}}

\newcommand{\ocinterval}[2]{(#1,\,#2]}%
\newcommand{\cointerval}[2]{[#1,\,#2)}%
\newcommand{\oointerval}[2]{(#1,\,#2)}%
\newcommand{\ccinterval}[2]{[#1,\,#2]}%

\newcommand{\even}{{\rm \,even}}

\newcommand{\tdots}{{...}}%
\newcommand{\supp}{{\rm supp}}

\newcommand{\pair}[2]{{\left({#1},\,{#2}\right)}}

\newcommand{\bskp}[2]{{\big\langle{#1},\,{#2}\big\rangle}}

\newcommand{\bpair}[2]{{\big({#1},\,{#2}\big)}}

\newcommand{\at}[1]{{\left({#1}\right)}}
\newcommand{\nat}[1]{(#1)}
\newcommand{\bat}[1]{{\big(#1\big)}}
\newcommand{\Bat}[1]{{\Big(#1\Big)}}

\newcommand{\ul}[1]{\underline{#1}}

\newcommand{\D}{\displaystyle}

\newcommand{\norm}[1]{\left\|{#1}\right\|}
\newcommand{\nnorm}[1]{\|{#1}\|}
\newcommand{\bnorm}[1]{\big\|{#1}\big\|}

\newcommand{\abs}[1]{\left|{#1}\right|}
\newcommand{\babs}[1]{\big|{#1}\big|}

\newcommand{\dint}[1]{\,\mathrm{d}#1}

\newcommand{\Om}{{\Omega}}
\newcommand{\al}{{\alpha}}
\newcommand{\be}{{\beta}}

\newcommand{\eps}{{\varepsilon}}

\newcommand{\la}{{\lambda}}
\newcommand{\si}{{\sigma}}

\newcommand{\om}{{\omega}}

\newcommand{\calA}{\mathcal{A}}
\newcommand{\calB}{\mathcal{B}}
\newcommand{\calC}{\mathcal{C}}
\newcommand{\calD}{\mathcal{D}}

\newcommand{\calF}{\mathcal{F}}
\newcommand{\calG}{\mathcal{G}}

\newcommand{\calK}{\mathcal{K}}
\newcommand{\calL}{\mathcal{L}}
\newcommand{\calM}{\mathcal{M}}
\newcommand{\calN}{\mathcal{N}}

\newcommand{\calP}{\mathcal{P}}
\newcommand{\calQ}{\mathcal{Q}}

\newcommand{\calT}{\mathcal{T}}

\makeatother

\theoremstyle{plain}
\newtheorem{theorem}             {Theorem}[]
\newtheorem{corollary}  [theorem]{Corollary}

\newtheorem{lemma}      [theorem]{Lemma}
\newtheorem{proposition}[theorem]{Proposition}

\newtheorem*{result*}{Main result}

\newtheorem{assumption} [theorem]{Assumption}

\usepackage{float}
\begin{document}

% -----------------------------------------------------------------------------
% - Title information
% -----------------------------------------------------------------------------

\title{Solitary waves in atomic chains
and peridynamical media}

\date{\today}

\author{%
Michael Herrmann\footnote{Technische Universit\"at Braunschweig, Computational Mathematics, {\tt michael.herrmann@tu-braunscheig.de}}
\and
Karsten Matthies\footnote{ University of Bath, Department of Mathematical Sciences, {\tt k.matthies@bath.ac.uk}}
 }

\maketitle

% -----------------------------------------------------------------------------
% - Abstract
% -----------------------------------------------------------------------------

\begin{abstract}
Peridynamics describes the nonlinear interactions in spatially
extended Hamiltonian systems by nonlocal integro-differential equations, which can be regarded as the natural generalization of lattice models. We prove the existence of solitary traveling waves for super-quadratic potentials by maximizing the potential energy
subject to both a norm and a shape constraint. We also discuss the numerical computation of waves and study several asymptotic regimes.
\end{abstract}

% -----------------------------------------------------------------------------
% - MSC and keywords
% -----------------------------------------------------------------------------
%
 \quad\newline\noindent%
 \begin{minipage}[t]{0.15\textwidth}%
   Keywords:
 \end{minipage}%
 \begin{minipage}[t]{0.8\textwidth}%
  \emph{peridynamics}, \emph{Hamiltonian lattices}, \emph{solitary waves},\\  \emph{variational methods}, \emph{asymptotic analysis}
 \end{minipage}%
 \medskip
 \newline\noindent
 \begin{minipage}[t]{0.15\textwidth}%
   MSC (2010): %
 \end{minipage}%
37K40,  %   Soliton theory, asymptotic behavior of solutions, in 37-xx Dynamical systems
37K60,  %	Lattice dynamics, in 37Kxx: Infinite-dimensional Hamiltonian systems%
47J30, %Variational methods (in 47Jxx=equations and inequalities involving nonlinear operators)
74J30, %Nonlinear waves ; n 74-xx=mechanics of deformable solids)
74H15  	%Numerical approximation of solutions, in 74Hxx: Dynamical problems for solids
 \begin{minipage}[t]{0.8\textwidth}%
 \end{minipage}%

% -----------------------------------------------------------------------------
% - Table of contents
% -----------------------------------------------------------------------------
%
\setcounter{tocdepth}{4}
\setcounter{secnumdepth}{3}{\scriptsize{\tableofcontents}}
%
%
%
% -------------------------------------------------------------------------------------
\section{Introduction}
\label{sect:Intro}
% -------------------------------------------------------------------------------------
%
%
Peridynamics is a modern branch of solid mechanics and materials science, which models the physical interactions in a material continuum not by partial differential equations but in  terms of nonlocal integro-differential equations, see \cite{Sil00,SL10} for an introduction and overview. The simplest dyna\-mical model for a spatially one-dimensional and infinitely extended medium is
\begin{align}
\label{Eqn:DynamicalSystem}
\partial_t^2 u\pair{t}{y}=\int\limits_{-\infty}^{+\infty} f\bpair{u\pair{t}{y+\xi}-u\bpair{t}{y}}{\xi}\dint\xi
\end{align}
with time $t>0$, material space coordinate $y\in\Rset$, bond variable \BMHR $\xi\in\Rset$, \EMHR and scalar displacement field $u$. The elastic force function $f$ is usually supposed to satisfy Newton's third law of motion via
\begin{align}
\label{Eqn:Newton3}
f\pair{r}{\xi}=-f\pair{-r}{-\xi}
\end{align}
so that \eqref{Eqn:DynamicalSystem} can equivalently be written as
\begin{align}
\label{Eqn:DynamicalSystem2}
\partial_t^2 u\pair{t}{y}=\int\limits_0^{\infty}f\bpair{u\pair{t}{y+\xi}-u\bpair{t}{y}}{\xi}-
f\bpair{u\pair{t}{y}-u\bpair{t}{y-\xi}}{\xi}\dint\xi\,.
\end{align}
The symmetry condition \eqref{Eqn:Newton3} further ensures that \eqref{Eqn:DynamicalSystem2} admits both a Lagrangian and Hamiltonian structure and can hence be regarded as a nonlocal wave equation without dissipation, see also \S\ref{sect:AsympLin}. For energetic considerations it is also useful to write
\begin{align*}
f\pair{r}{\xi}=\partial_r \Psi\pair{r}{\xi}\,,
\end{align*}
where the micro-potential $\Psi\pair{r}{\xi}$ quantifies the contribution to the potential energy coming from the elastic deformation of the bond $\xi$. The corresponding energy density is given by
\begin{align*}
e\pair{t}{y}=\tfrac12 \bat{\partial_t u\pair{t}{y}}^2+\int\limits_0^\infty
\Psi\pair{u\pair{t}{y+\xi}-u\pair{t}{y}}{\xi}\dint\xi
\end{align*}
and the total energy is conserved according to $\frac{\dint}{\dint{t}}\int_{-\infty}^{+\infty} e \pair{t}{y}\dint{y}=0$.
%
%
%-------------------------------------------------------------------------------------
\paragraph{Peridynamics and lattice models}
%-------------------------------------------------------------------------------------
%
%
In this paper we focus on two particular settings. In the simplified \emph{continuous-coupling case} we have
\begin{align}
\label{Case:ContinousCoupling}
\Psi\pair{r}{\xi}=\al\at\xi \Phi\bat{\beta\at\xi r}
\end{align}
with smooth reference potential $\Phi:\Rset\to\Rset$ and scaling coefficients provided by
two sufficiently regular functions  $\al,\beta:\Rset_+\to \Rset_+$. This is a typical choice in peridynamics and \cite{Sil16} proposes for instance
\begin{align}
\label{Eqn.Silling}
a\at{\xi}=\chi_{\ccinterval{0}{H}}\at{\xi}\xi\,,\qquad
\beta\at{\xi}=\chi_{\ccinterval{0}{H}}\at{\xi}\xi^{-1}\,,\qquad \Phi\at{r}=
\left\{\begin{array}{lcl}
c_2 r^2+c_3r^3 &&\text{for $r\leq0$,}\\
c_2 r^2&&\text{for $r>0$,}\\
\end{array}
\right.
\end{align}
where $H$ is the \emph{horizon} and $\chi_J$ is shorthand for the indicator function of the interval $J$.
\par
The \emph{discrete-coupling case} corresponds to
\begin{align}
\label{Case:DiscreteCoupling}
\Psi\pair{r}{\xi}=\sum_{m=1}^M \Phi_m\at{r}\delta_{\xi_m}\at\xi\,,
\end{align}
where
\begin{align*}
0<\xi_1<\tdots<\xi_M<\infty
\end{align*}
represent a finite number of active bonds and $\delta_{\xi_m}$ abbreviates a Dirac distribution centered at $\xi_m$. The integrals in the force term can hence be replaced by sums and the wave equation \eqref{Eqn:DynamicalSystem2} is actually a nonlocal lattice differential equation. For instance, for $M=1$ and $\xi_1=1$ we recover a variant of the classical Fermi-Pasta-Ulam-Tsingou chain (FPUT) while $M=2$ with \BMHC $\xi_1=1$ and  $\xi_2=2$ \EMHC describe an atomic chain with spring-like bonds between the nearest and the next-to-nearest neighbors.
\par
Of course, both the discrete-coupling case and the continuous-coupling case are closely related and \eqref{Eqn:DynamicalSystem2} can be viewed as a generalized lattice equation with a continuum of active bonds. Moreover, \eqref{Eqn:DynamicalSystem2}+\eqref{Case:DiscreteCoupling} with
\begin{align}
\label{Eqn:TwoSettings}
\Phi_m\at{x}=\al\at{\eps m}\Phi\bat{\beta\at{\eps m} r},\,\qquad M=\eps^{-2}
\end{align}
is a discretized version of \eqref{Eqn:DynamicalSystem2}+\eqref{Case:ContinousCoupling}, in which the integrals have been approximated by Riemann sums with grid size $\eps>0$ on the interval $\oointerval{0}{\eps^{-1}}$.
%
%
%-------------------------------------------------------------------------------------
\paragraph{Traveling waves and eigenvalue problem for the wave profile}
%-------------------------------------------------------------------------------------
%
%
A traveling wave is a special solution to \eqref{Eqn:DynamicalSystem2} that satisfies
\begin{align}
\label{Eqn:TWAnsatz}
u\pair{t}{y}=U\at{x}\,,\qquad x=y-\si t
\end{align}
with wave speed $\si$ and profile function $U$ depending on $x$, the \BMHC spatial \EMHC variable in  the comoving frame. \BMHC Combining \EMHC this ansatz with \eqref{Eqn:DynamicalSystem2}
we  obtain in the discrete-coupling case the advance-delay-differential equation
\begin{align}
\label{Eqn:TWLattice}
\si^2 U^{\prime\prime}\at{x}=\sum_{m=1}^M
\Phi_m^\prime\at{U\bat{x+\xi_m}-U\at{x}}-\Phi_m^\prime\at{U\bat{x}-U\at{x-\xi_m}}
\end{align}
and a similar formula with integrals over infinitely many shift terms can be derived in the continuous-coupling case. As explained below in greater detail, the existence of solution to \eqref{Eqn:TWLattice} with $M=1$ has been established by several authors using different methods but very little is known about the uniqueness and dynamical stability of lattice waves. For the peridynamical analogue with continuous coupling we are only \BMHC aware \EMHC of the rigorous existence result in \cite{PV18}  although there have been some attempts in the engineering community to construct traveling waves numerically or approximately, see \cite{DB06,Sil16}.
\par
For our purposes it is more convenient to reformulate the traveling wave equation as an integral equation for the \emph{wave profile}
\begin{align*}
W\at{x}:=U^\prime\at{x}\,,
\end{align*}
which provides the velocity component in a traveling wave via $\partial_t u\pair{t}{y}=\si W\at{x}$. To this end we write
\begin{align}
\label{Eqn:ToolA1}
f\pair{U\bat{x+\xi}-U\at{x}}{\xi}-f\pair{U\bat{x}-U\at{x-\xi}}{\xi}=\partial_x F\pair{x}{\xi}
\end{align}
with
\begin{align}
\label{Eqn:ToolA2}
F\pair{\cdot}{\xi}
:=
\calA_\xi \partial_r\Psi\pair{\calA_\xi W}{\xi}\,,
\end{align}
where the operator $\calA_\xi$ defined by
\begin{align}
\label{Eqn:ConvOperatorA}
\bat{\calA_\xi W}\at{x}:=\int\limits_{x-\xi/2}^{x+\xi/2} W\at{\tilde{x}}\dint{\tilde{x}}=\bat{\chi_\xi\ast W}\at{x}
\end{align}
is the $x$-convolution with the \BMHC indicator  \EMHC  function $\chi_\xi$ of the interval $\ccinterval{-\xi/2}{+\xi/2}$. Thanks to \eqref{Eqn:ToolA1}  and \eqref{Eqn:ToolA2} we infer from \eqref{Eqn:DynamicalSystem2} and \eqref{Eqn:TWAnsatz} that any traveling waves must satisfy the nonlinear eigenvalue problem
\begin{align}
\label{Eqn:TW1}
\si^2 W = \int\limits_0^\infty\calA_\xi\partial_r\Psi\pair{\calA_\xi W}{\xi}{\dint\xi}+\eta\,,
\end{align}
where $\eta$ is a constant of integration. The eigenvalue problem
\eqref{Eqn:TW1} is very useful for analytical investigations because the convolution operator \eqref{Eqn:ConvOperatorA} exhibits some nice invariance properties. Moreover, it also gives rise to an approximation scheme that can easily be implemented and shows --- at least in practice --- good convergence properties. Both issues will  be discussed below in greater detail \BMHR and many \EMHR of the key arguments have already been exploited in the context of FPUT chains, see \cite{FV99,Her10}.
\par
To simplify the presentation, we restrict our considerations in this paper to the special case
\begin{align}
\label{Eqn:TWX}
\eta=0
\end{align}
but emphasize that this condition can always be guaranteed by means of elementary transformations applied to the wave profile and the micro-potentials, see for instance \cite{Her10}. For solitary waves one can alternatively relate \BMHC $\eta$  \EMHC to waves that are homoclinic with respect to a non-vanishing asymptotic state $\lim_{x\to\pm\infty}W\at{x}\neq 0$.
\begin{figure}[t!]%
\centering{%
\includegraphics[width=.9\textwidth]{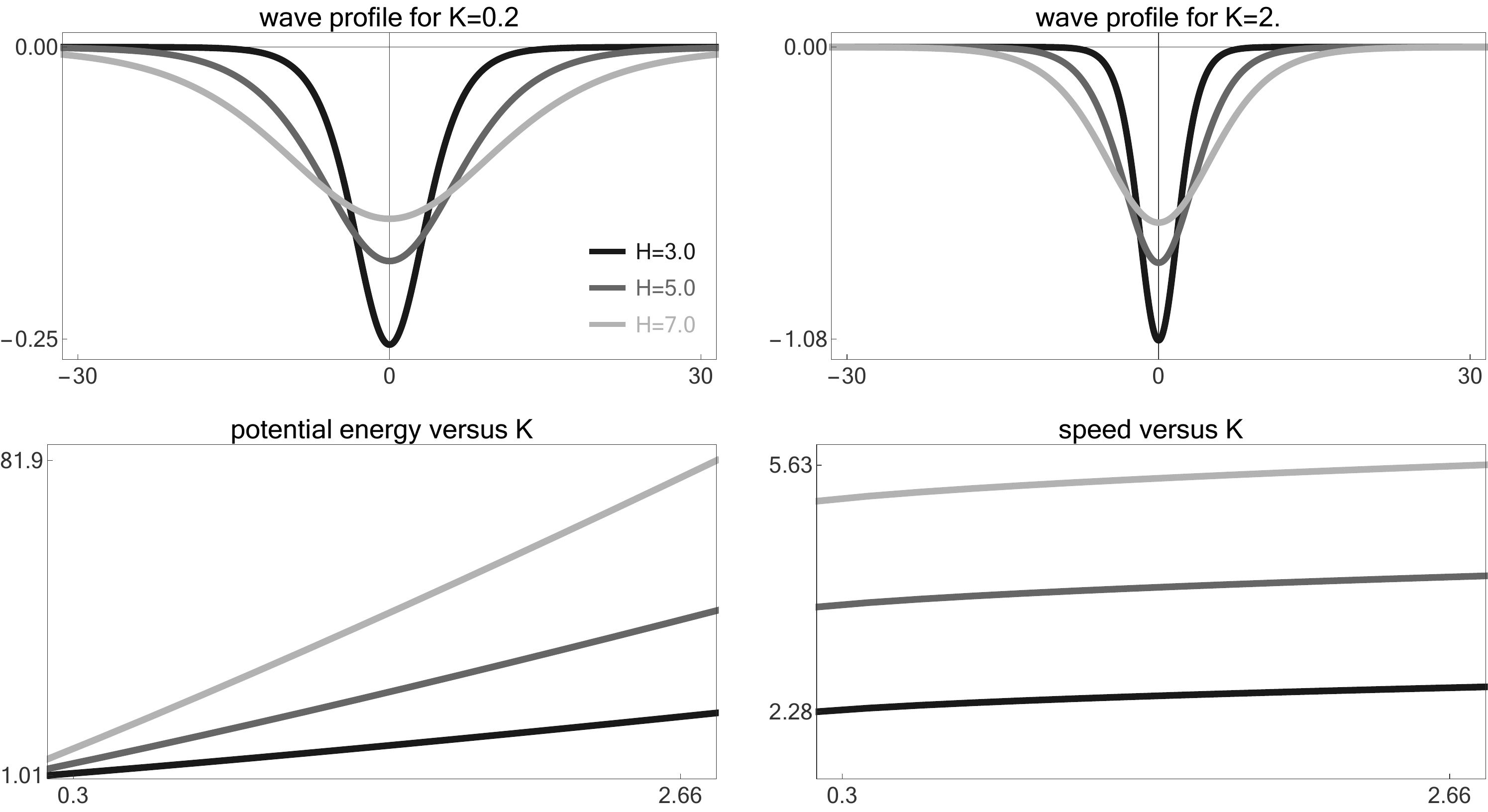}
}%
\caption{\emph{Top row}. Numerical simulations of wave profiles for  \eqref{Eqn.Silling} with $c_1=1/2$, $c_2=1/6$, and several choices of the coupling horizon $H$. The data are computed by an implementation of the improvement dynamics \eqref{Eqn:ImprovementDynamics} with negative initial profile,  sufficiently small discretization parameter $\eps$ and sufficiently large periodicity length $L$. Notice that the profile functions are non-positive as in \cite{Sil16} so that the analytical results in \S\ref{sect:Existence} hold with respect to the reflected reference potential.
\emph{Bottom row}. Potential energy $\calP\at{W}$ and wave speed $\si$ in dependence of the constraint value $K=\calK\at{W}$.
}%
\label{fig:silling}%
\end{figure}%
%
%
%
%-------------------------------------------------------------------------------------
\paragraph{Variational setting for solitary and periodic waves}
%-------------------------------------------------------------------------------------
%
%
In this paper we construct solitary waves for systems with convex and super-quadratic interaction potential but for completeness we also discuss the existence of periodic waves. Other types of traveling waves or potentials are also very important but necessitate a more sophisticated analysis and are beyond the scope of this paper. Examples are heteroclinic waves or phase transition waves with oscillatory tails in systems with double-well potential as studied in \cite{HR10,Her11,TV05,DB06,SZ12,HMSZ13} for atomic chains.
\par
For periodic waves we fix a length parameter $0<L<\infty$ and look for $2L$-periodic wave profiles that are square-integrable on the periodicity cell. This reads
\begin{align*}
W\in\fspaceL^2\bat{I_L}:=\Big\{W:\Rset\to\Rset\quad\text{with}\quad W\at{\cdot}=W\at{\cdot+2L}\quad\text{and}\quad\norm{W}_2<\infty\Big\}
\end{align*}
with
\begin{align*}
 \norm{W}_2^2:=
\int\limits_{I_L}W\at{x}^2\dint{x}\,,\qquad I_L:=\oointerval{-L}{+L}\,.
\end{align*}
For solitary waves we formally set $L=\infty$ and show in \S\ref{sect:ConvergencePeriodicWaves} that the corresponding waves can in fact be regarded as the limit of periodic waves as $L\to\infty$. \BMHR We also
mention that any $\calA_\xi$ from \eqref{Eqn:ConvOperatorA} is a well-defined, bounded, and symmetric operator on $\fspaceL^2\at{I_L}$, see Lemma \ref{Lem.PropsA} below. \EMHR
\par
Both for periodic and solitary waves we consider the potential energy functional
\begin{align*}
\calP\at{W}:=\int\limits_0^\infty \int\limits_{I_L}\Psi\bpair{\at{\calA_\xi W}\at{x}}{\xi}\dint{x}\dint\xi
\end{align*}
as well as the functional
\begin{align*}
\calK\at{W}:=\tfrac12 \norm{W}_2^2\,,
\end{align*}
where the latter quantifies after multiplication with $\si^2$ the integrated kinetic energy \BMHC of \EMHC a traveling wave. Using these energies we can reformulate the traveling wave equation \eqref{Eqn:TW1} with \eqref{Eqn:TWX} as
\begin{align}
\label{Eqn:TW2}
\si^2 \partial\calK\at{W} =\partial\calP\at{W}\,,
\end{align}
where $\partial$ denotes the G\^ateaux derivative with respect to $W$ and the inner product in $\fspaceL^2\at{I_L}$. In  particular, in the discrete-coupling case we have
\begin{align}
\label{Eqn:EnergyFormulas1}
\calP\at{W}=\sum_{m=1}^M\int\limits_{I_L}\Phi_m\bat{\calA_{\xi_m} W}\dint{x}\,,\qquad
\partial\calP\at{W}= \sum_{m=1}^M\calA_{\xi_m}\Phi^\prime_m\bat{\calA_{\xi_m} W}\,
\end{align}
while the continuous-coupling case corresponds to
\begin{align}
\label{Eqn:EnergyFormulas2}
\calP\at{W}=\int\limits_{0}^{\infty}\int\limits_{I_L}\al\at\xi\Phi\bat{\be\at\xi\calA_\xi W}\dint{x}\BMHC \dint\xi \EMHC\,,\qquad
\partial\calP\at{W}=\int\limits_{0}^{\infty}\al\at\xi
\be\at\xi\calA_\xi\BMHR \Phi^\prime\EMHR\bat{\be\at\xi \calA_\xi W}\dint\xi\,,
\end{align}
where we omitted the $x$-dependence of $W$ and $\calA_\xi W$ to ease the notation.
%
%-------------------------------------------------------------------------------------
\paragraph{Existence result}
%-------------------------------------------------------------------------------------

The existence of periodic and solitary traveling waves in FPUT chains
has been studied intensively over the last two decades and any of the proposed methods can also be applied to peridynamical media although the discussion of the technical details might be more involved.
\begin{enumerate}
\item %
The Mountain Pass Theorem allows to construct nontrivial critical points of the
Lagrangian action functional
\begin{align*}
W\mapsto\si^2\calK\at{W}-\calP\at{W}
\end{align*}
provided that the potential energy is super-quadratic and that the prescribed value of the wave speed $\si$ is sufficiently large. For details in the context of FPUT chains we refer to \cite{Pan05} and the references therein.
\item
Another variational setting minimizes $\calK$ subject to prescribed value of $\calP$, where $\si^{-2}$ plays the role of \BMHR a \EMHR \BMHC Lagrange \EMHC multiplier. This approach was first described in \cite{FW94} for FPUT chains with super-quadratic
potentials and has recently been generalized to peridynamical media in \cite{PV18}.
\item
It is also possible \BMHC to \EMHC construct traveling waves by maximizing $\calP$ under the constraint of prescribed $\calK$. This idea was introduced in \cite{FV99},  has later been refined in \cite{Her10} and provides also the base for our approach as it allows to impose  additional shape constraints as discussed in  \S\ref{sect:Existence}.  Moreover,  a homogeneous constraint is \BMHC more easily \EMHC imposed than a non-homogeneous one and this simplifies the numerical computation of traveling waves.
\item
The concepts of spatial dynamics and center-manifold reduction \BMHC have \EMHC been exploited
in \cite{IK00,IJ05} for lattice waves. This non-variational approach is restricted to small amplitude waves but provides --- at least in principle --- a complete picture on all solutions to \eqref{Eqn:TW1}. Moreover, explicit or approximate solutions are available
for some special potentials or asymptotic regimes. We refer to \cite{TV14,SV18} and the more detailed discussion in \S\ref{sect:AsympLin}.
\end{enumerate}
Our key findings on the existence of peridynamical waves can informally be summarized as follows. The precise statements can be found in Assumption \ref{Ass.AdmissiblePhi}, Corollary \ref{Cor:ExistenceSolitaryWaves}, Proposition \ref{Prop.SQCriterion}, and
Proposition \ref{Prop:ConvergencePeriodicWaves}.

\begin{result*}[existence of solitary waves]
Suppose that the coefficients and interaction potentials in \eqref{Case:ContinousCoupling} or \eqref{Case:DiscreteCoupling}
\BMHC are sufficiently smooth and satisfy natural integrability  \EMHC and super-quadraticity assumptions. Then there exists
a family of solitary waves which is parameterized by $K=\calK\at{W}$ and comes with unimodal, even, and nonnegative profile functions $W$. Moreover, each solitary wave
can be approximated by periodic ones.
\end{result*}
Notice that almost nothing is known about the uniqueness or dynamical stability
of solitary traveling waves in peridynamical media and it remains a very challenging task
solve the underlying mathematical problems. In \S\ref{sect:AsympLin} we thus discuss the asymptotic regimes of near-sonic and high-speed waves, in which one might hope for first rigorous results in this context.
%
%
%
%-------------------------------------------------------------------------------------
\paragraph{Improvement dynamics and numerical simulations}
%-------------------------------------------------------------------------------------
%
%
A particular ingredient to our analysis is the \emph{improvement operator}
\begin{align}
\label{Eqn:ImprovementOperator}
\calT\at{W}:=\mu\at{W}\,\partial\calP\at{W}\,,\qquad \mu\at{W}:=\norm{W}_2/\norm{\partial\calP\at{W}}_2\,,
\end{align}
which conserves the norm constraint $\calK$ but increases the variational objective function $\calP$, see  Proposition \ref{Prop:ImprovementDynamics}. Moreover,
the operator $\calT$ respects \BMHR --- under \EMHR natural convexity and normalization assumptions on the involved micro-potentials \BMHR --- the \EMHR unimodality, the evenness, and the positivity (or negativity) of $W$. \BMHC This \EMHC enables us in \S\ref{sect:Existence} to restrict the constrained optimization problem to a certain shape cone in $\fspaceL^2\at{I_L}$ without changing the Euler-Lagrange equation for \BMHC maximizers,  \EMHC see  Corollary \ref{Cor:MaximizerAreWaves}.
\par%
By iterating \eqref{Eqn:ImprovementOperator} in the discrete dynamics
\begin{align}
\label{Eqn:ImprovementDynamics}
W_0\quad\rightsquigarrow\quad W_1=\calT\at{W_0}\quad\rightsquigarrow\quad\tdots
\quad\rightsquigarrow\quad W_{n+1}=\calT\at{W_n}\quad\rightsquigarrow\quad\tdots
\end{align}
we can also construct sequences $\at{W_n}_{n\in\Nset}$ of profiles functions such that
$\calK\at{W_n}$ remains conserved while $\calP\at{W_n}$ is bounded and strictly increasing in $n$. Using weak compactness for periodic waves as well as a variant of concentration compactness for solitary waves one can also show that there exist strongly convergent subsequences and that any accumulation point must be a traveling wave, see the proofs of Lemma \ref{Lem:PeriodicMaximizers} and \BMHC Theorem \EMHC \ref{Prop:SolitaryCompactness}. However, due to the lack of uniqueness results we are not able to prove the uniqueness of accumulation points \BMHC or \EMHC the convergence of the whole sequence. Moreover, the set of accumulation points might depend on the choice of the initial datum $W_0$.
\par
The improvement \BMHC mapping \EMHC  is nonetheless very useful for numerical purposes. Fixing a periodicity length $L$ it can easily bee implemented:
\begin{enumerate}
\item
divide the $x$-domain $I_L$ and the $\xi$-domain $\cointerval{0}{\infty}$ into a large but finite number of subintervals of length $0<\eps\ll 1$,  and
\item
replace all integrals with respect to $x$ and $\xi$ by Riemann sums as in \eqref{Eqn:TwoSettings}.
\end{enumerate}
The resulting scheme has been used to compute the numerical data presented in this paper. It shows very good and robust convergence properties in practice, and this indicates for a wide class of peridynamical media that there exists in fact a unique stable solution to the constrained optimization problem. A first example is shown in Figure \ref{fig:silling} and corresponds to the peridynamical medium \eqref{Eqn.Silling}, for which approximate solitary waves are constructed in \cite{Sil16} by means of formal asymptotic expansions and ODE arguments. Notice that $W_0$ was always chosen to be non-positive in order to pick up the super-quadratic branch in the potential and that there are no solitary waves extending into the harmonic branch, \BMHR cf. the discussion at the beginning of \S\ref{sect:Superquadraticity}.\EMHR
\par
\BMHR We finally mention that variants of \eqref{Eqn:ImprovementDynamics} have already been used in \cite{FV99, EP05, Her10} for the numerical computations of periodic waves in FPUT chains. Moreover, the improvement dynamics shares some similarities with the Petviashvili iteration for traveling waves in Hamiltonian PDEs, see for instance \cite{PS04}. Both approximation schemes combine linear pseudo-differential operators with pointwise nonlinearities and a dynamical normalization rule, but the details are rather different. In particular, \eqref{Eqn:ImprovementDynamics} comes with a Lyapunov function but does not allow to prescribe the wave speed.
\EMHR
%
%
%
%-------------------------------------------------------------------------------------
\section{Variational approach for super-quadratic potentials}\label{sect:Existence}
%-------------------------------------------------------------------------------------
%
 \BMHC We  now \EMHC study the existence problem for a  class of potentials in a certain cone of $\fspaceL^2\at{I_L}$ and  \BMHC begin by \EMHC clarifying the precise setting.
%
%-------------------------------------------------------------------------------------
\subsection{Setting of the problem}
%-------------------------------------------------------------------------------------
%
A smooth function $\Phi:\cointerval{0}{\infty}\to\cointerval{0}{\infty}$ is called \emph{weakly super-quadratic} if it does not vanish identically
and satisfies
\begin{align}
\label{Eqn:Superquadraticty}
\Phi\at{0}=\Phi^\prime\at{0}=0\,,\qquad \quad
\Phi^{\prime\prime}\at{r}r\geq  \Phi^\prime\at{r}\geq 0\qquad\text{for all}\quad r\geq0\,.
\end{align}
Such a function is increasing, convex, and grows at least quadratically since elementary arguments  \BMHC --- including differentiation with respect to $\lambda$ --- \EMHC show that
\begin{align}
\label{Eqn:PropertiesPhi} \Phi^\prime\at{r}r\geq 2\Phi\at{r}\,, \qquad\quad
\Phi\at{\BMHC \la \EMHC r}\geq \la^2\Phi\at{r}\quad \text{for}\quad \la>1\,,
\end{align}
Examples are the harmonic potential $\Phi\at{r}=cr^2$ and
any analytic function with non-negative Taylor-coefficients. In what follows we rely on the following standing assumption.
\begin{assumption}[admissible potentials]
\label{Ass.AdmissiblePhi}
In the discrete-coupling case, each potential $\Phi_m$ is  supposed to be super-quadratic in the sense of \eqref{Eqn:Superquadraticty}. In the continuous-coupling \BMHR case, we \EMHR assume that $\Phi$ is super-quadratic and that $\alpha$, $\beta$ are nonnegative functions on $\oointerval{0}{\infty}$ such that
\begin{align*}
\int\limits_0^\infty \bat{\xi+\xi^{2}}\al\at\xi\beta^2\at\xi \phi\bat{\xi^{1/2}\,\beta\at\xi\,\sqrt{2K}} \dint\xi<\infty
\end{align*}
holds for all $K>0$ with $\phi\at{r}:=r^{-1}\babs{\Phi^\prime\at{r}}$.  \BMHR Moreover, in the proofs we suppose that each potential is three times continuously differentiable on the interval $\cointerval{0}{\infty}$. \EMHR
\end{assumption}
Both for finite and infinite $L$, we denote by $\calC$ the positive cone of all $\fspaceL^2\at{I_L}$-functions that are even, nonnegative and unimodal, where the latter means increasing and decreasing for $-L<x<0$ and $0>x>L$, respectively.  This reads
\begin{align*}
\calC:=\ol{\{W\in \fspace{C}^\infty_{\text{c}}\at{I_L}\;:\;
\text{$W\at{x}=W\at{-x}\geq0$ and $-W^\prime\at{x}=W^\prime\at{-x}\geq0$ for all $x\geq0$}
\}}\,,
\end{align*}
where the closure has to be taken with respect to the norm in $\fspaceL^2\at{I_L}$, and we readily verify that $\calC$ is a convex cone and closed  under both strong and weak convergence. Moreover, the decay estimate
\begin{align}
\label{Eqn:UnimodalDecay}
0\leq W\at{x}\leq \frac{\norm{W}_2}{\sqrt{2\abs{x}}}
\end{align}
holds for any $W\in\calC$ and any $x\in I_L$.
\par
We  next collect some important properties of the convolution operators $\calA_\xi$ and recall that $\mathrm{sinc}\at{y}:=y^{-1}\sin\at{y}$.
\begin{lemma}[properties of $\calA_\xi$]
\label{Lem.PropsA}
For given $\xi\in\oointerval{0}{\infty}$, the following statements are satisfied for both finite and infinite $L$:
\begin{enumerate}
\item
$\calA_\xi$ maps $\fspaceL^2\nat{I_L}$ to
$\fspaceL^2\nat{I_L}\cap \fspaceL^\infty\bat{I_L}$ with $\norm{\calA_\xi W}_2\leq  \xi\norm{W}_2$ and
$\norm{\calA_\xi W}_\infty\leq  \xi^{1/2}\norm{W}_2$.
\item
$\calA_\xi$ is symmetric.
\item
$\calA_\xi$ diagonalizes in Fourier space and has symbol function $\widehat{\calA}_\xi\at{k}=\xi\mathrm{sinc}\at{k\xi/2}$.
\item
$A_\xi$ respects the evenness, the non-negativity, and the unimodality
of functions.
\end{enumerate}
Moreover,  $A_\xi$ is compact for $L<\infty$ as it maps weakly converging into strongly converging sequences.
\end{lemma}
\begin{proof}
All assertions follow from elementary computations and standard arguments, see for instance \cite[Lemma 2.5]{Her10}.
\end{proof}
We are now able to show that our assumptions and definitions allow for a consistent setting of the variational traveling wave problem inside the cone $\calC$.
\begin{lemma}[properties of the potential energy functional]
\label{Lem:PotentialEnergy}
The functional $\calP$ is well-defined, strongly continuous and G\^ateaux differentiable on the cone $\calC$, where the derivative maps $\calC$ \BMHR continously \EMHR into itself.  Moreover, $\calP$ is convex, super-quadratic in the sense of
\begin{align}
\label{Lem:PotentialEnergy.E1}
\bskp{\partial\calP\at{W}}{W}\geq 2\calP\at{W}\geq0\,,\qquad
\calP\at{\la W}\geq \la^2\calP\at{W}\quad \text{for \; $\la\geq1$ }\,,
\end{align}
and  non-degenerate as $\partial\calP\at{W}\neq 0$ holds for any $W\in\calC$ with $\calP\at{W}>0$.
\end{lemma}
\begin{proof}
We present the proof for the continuous-coupling case \BMHR \eqref{Case:ContinousCoupling}  only \EMHR but emphasize that the arguments  for \eqref{Case:DiscreteCoupling} are similar.
\par
\emph{\ul{Integrability}}: For given $W\in\calC$ with $\calK\at{W}=K$ we estimate
\begin{align*}
0\leq \Phi^\prime\bat{\beta\at\xi\calA_\xi W}\leq \be\at\xi\at{\sup_{0\leq r\leq \beta\at\xi\norm{\calA_\xi W}_\infty}\phi\at{r}}\calA_\xi W \leq
\be\at\xi \phi\at{\xi^{1/2}\beta\at\xi\sqrt{2K}}\calA_\xi W
\end{align*}
with $\phi$ as in Assumption \ref{Ass.AdmissiblePhi}, where we used that the super-quadraticity of $\Phi$ implies the monotonicity of $\phi$. The estimates from Lemma \ref{Lem.PropsA} imply via
\begin{align*}
\norm{\calA_\xi \Phi^\prime\bat{\beta\at\xi\calA_\xi W}}_2\leq \xi^2 \be\at\xi \phi\at{\xi^{1/2}\beta\at\xi\sqrt{2K}}\norm{W}_2
\end{align*}
the estimate
\begin{align*}
\norm{\partial\calP\at{W}}_2\leq \int\limits_0^\infty \al\at\xi\BMHR\beta\at\xi \EMHR\bnorm{
\calA_\xi \Phi^\prime\bat{\beta\at\xi\calA_\xi W}}_2\dint\xi<\infty\,,
\end{align*}
and similarly we show $\calP\at{W}<\infty$ using \eqref{Eqn:PropertiesPhi}. Moreover, since $\calC$ is invariant under $\calA_\xi$, multiplication with non-negative constants, and the \BMHC composition \EMHC with $\Phi^\prime$, we easily show that $\partial\calP\at{W}\in\calC$.
\par
\BMHR \emph{\ul{Continuity}}: Let $\at{W_n}_{n\in\Nset}\subset \calC$ be a sequence that converges strongly in $\fspaceL^2\at{I_L}$ to some limit $W_\infty$, where the closedness properties of $\calC$ impliy $W_\infty\in\calC$. Extracting a subsequence we can assume that $W_n$ converges pointwise to $W_\infty$ and that there exists a dominating function $\ol{W}\in\fspaceL^2\at{I_l}$ such that $\abs{W_n\at{x}}\leq \ol{W}\at{x}$ for all $n\in\Nset$ and almost all $x\in I_L$. Using the similar arguments as above we readily demonstrate  that the functions
\begin{align*}
V_n\pair{\xi}{\cdot}:=\al\at\xi\beta\at\xi\calA_\xi\Phi^\prime\bat{\beta\at\xi\calA_\xi W_n}
\end{align*}
converge  pointwise in $\pair{\xi}{x}$ to $V_\infty$ and are moreover dominated by a function $\ol{V}\in\fspaceL^1\bat{\oointerval{0}{\infty},\fspaceL^2\at{I_L}}$. Lebesgue's theorem ensures the strong convergence
$\partial\calP\at{W_n}\to \partial\calP\at{W_\infty}$ in $\fspaceL^2\at{I_L}$ along the choosen subsequence and due to the uniqueness of the accumulation point we finally establish the convergence of the entire sequence by standard \BMHC arguments. \EMHC  The proof of $\calP\at{W_n}\to \calP\at{W_\infty}$ is analogous. \EMHR
\par
\emph{\ul{Further properties}}:
The convexity and super-quadraticity of $\calP$ is a direct consequence of the corresponding properties of $\Phi$, and the convexity inequality
\begin{align}
\label{Eqn:ConvexityInequality}
\calP\bat{\widetilde{W}}-\calP\at{W}\geq \bskp{\partial\calP\at{W}}{\widetilde{W}-W}
\end{align}
ensures that any critical point of $\calP$ must be a global minimizer. In particular, we find the implication
\begin{align}
\label{Eqn:Nondegenericity}
\calP\at{W}>0=\calP\at{0}=\min\calP\quad\implies \quad \partial\calP\at{W}\neq0
\end{align}
and the proof is complete.
\end{proof}
We emphasize that all arguments presented  below hold analogously with $W\in -\calC$ for potentials that are super-quadratic on $\ocinterval{-\infty}{0}$, see Figure \ref{fig:silling} for an application. If $\Phi$ is super-quadratic on $\Rset$, our results imply under certain conditions the existence of two family of waves, which are usually called expansive ($W\geq0)$ or compressive ($W\leq0$).
%
%
%-------------------------------------------------------------------------------------
\subsection{Solutions of the constrained optimization problem} %-------------------------------------------------------------------------------------
\BMHR In this paper, \EMHR we prove the existence of traveling waves with unimodal profile functions \BMHC  for $0 <L \leq \infty $ \EMHC by solving the constrained optimization problem
\begin{align}
\label{Eqn:OptProblem}
\text{maximize $\calP\at{W}$ subject to $W\in\calC_K$}
\end{align}
with
\begin{align}
\label{Eqn:DefCK}
\calC_K:=\{W\in\calC\;:\calP\at{W}>0\text{ and } \calK\at{W}=K\}\,.
\end{align}
We further set
\begin{align}
\label{Eqn:DefPK}
P\at{K}:=\sup\big\{\calP\at{W}\;:\,W\in\calC_K\big\}
\end{align}
\BMHC  and \EMHC  recall that $\si^2$, the square of wave speed,  can be regarded as the \BMHC  Lagrange \EMHC  multiplier to the norm constraint $\calK\at{W}=K$. The second important observation, which we next infer \BMHC  from \EMHC  the properties of the improvement dynamics,  is that the shape constraint $W\in\calC$ does not contribute to the Euler-Lagrange equation for solutions to \eqref{Eqn:OptProblem}. The analogous observation for FPUT chains has been reported in \cite{Her10} and a similar result has been derived in \cite{KS12} using different techniques.
\begin{proposition}[properties of the improvement dynamics]
\label{Prop:ImprovementDynamics}
The set $\calC_K$ \BMHC in \EMHC   \eqref{Eqn:DefCK}  is invariant under the action of the improvement operator $\calT$ \BMHC in \EMHC  \eqref{Eqn:ImprovementOperator} and we have
\begin{align*}
\calP\bat{\calT\at{W}}\geq \calP\at{W}
\end{align*}
for any $W\in\calC_K$, where \BMHC  equality \EMHC  holds if and only if $W=\calT\at{W}$.
\end{proposition}
\begin{proof}
The operator $\calT$ is well defined on $\calC$ due to \eqref{Eqn:Nondegenericity} and the claimed invariance property  is a direct consequence of Lemma~\ref{Lem:PotentialEnergy}. The convexity inequality \eqref{Eqn:ConvexityInequality} implies
\begin{align}
\label{Eqn:ImprovementEstimate}
\calP\bat{\calT\at{W}}-\calP\at{W}&\geq \bskp{\partial\calP\at{W}}{\calT\at{W}-W}
\notag\\&=%
\mu^{-1}\at{W} \bskp{\calT\at{W}}{\calT\at{W}-W}
\notag\\&=%
\tfrac12\mu^{-1}\at{W}\bnorm{\calT\at{W}-W}^2\,,
\end{align}
where we used  $\mu\at{W}>0$ and that $\norm{W}_2=\norm{\calT\at{W}}_2$ holds by construction.
\end{proof}
\begin{corollary}
\label{Cor:MaximizerAreWaves}
Any solution $W\in\calC_K$ to the optimization problem \eqref{Eqn:OptProblem} satisfies the traveling wave equation \eqref{Eqn:TW2} for some \BMHC  Lagrange \EMHC  multiplier $\si^2\geq K^{-1}P\at{K}>0$.
\end{corollary}
\begin{proof}
Any maximizer $W$ belongs to $\calC_K$, so Proposition \ref{Prop:ImprovementDynamics} ensures in combination with \eqref{Eqn:ImprovementOperator} the validity of \eqref{Eqn:TW2} with $\si^2=\mu^{-1}\at{W}$. Testing this equation with $W$ we finally obtain
\begin{align*}
\si^2 2K =\bskp{\si^2 W}{W}= \bskp{\partial\calP\at{W}}{W}\geq 2\calP\at{W}=2P\at{K}
\end{align*}
thanks to \eqref{Lem:PotentialEnergy.E1}.
\end{proof}
It remains to show the existence of maximizers and here we have to distinguish between the periodic case ($L<0$) and the solitary one ($L=\infty$). The latter is more involved and will be discussed in the subsequent section.
\begin{lemma}[existence of periodic maximizers]
\label{Lem:PeriodicMaximizers}
For any $K>0$ and any $0<L<\infty$, the constrained optimization problem \eqref{Eqn:OptProblem} admits at least one solution.
\end{lemma}
\begin{proof}
Since the discrete coupling-case is rather simple, we discuss the continuous-coupling case only.
\par
\emph{\ul{Weak continuity of $\calP$}}: Suppose that $\at{W_n}_{n\in\Nset}\subset\calC_K$ converges weakly to $W_\infty$ and notice that the properties of convolution operators imply the pointwise convergence $V_n\pair{x}{\xi}\to V_\infty\pair{x}{\xi}$ for any $x\in I_L$ and any $\xi$ with
\begin{align*}
V_n\pair{x}{\xi}:=\al\at{\xi}\Phi\bat{\beta\at\xi \at{\calA_\xi W_n}\at{x}}\,.
\end{align*}
By \eqref{Eqn:PropertiesPhi} and Lemma \ref{Lem.PropsA} we also have
\begin{align*}
\babs{ \at{\calA_\xi W_n}\at{x}
}\leq \xi^{1/2}\norm{W_n}_2
\end{align*}
as well as
\begin{align*}
0\leq V_n\pair{x}{\xi}\leq \tfrac12 \al\at\xi \phi\bat{\xi^{1/2}\sqrt{2K}}\Bat{\beta\at\xi \at{\calA_\xi W_n}\at{x}}^2\leq
\xi\al\at\xi \beta^2\at\xi \phi\bat{\xi^{1/2}\sqrt{2K}} K\,,
\end{align*}
and since the right hand side is an integrable majorant on $I_L\times \cointerval{0}{\infty}$ according to Assumption \ref{Ass.AdmissiblePhi}, we conclude that $\calP\at{W_n}\to\calP\at{W_\infty}$.
\par
\emph{\ul{Direct Method}}: Now let $\at{W_n}_{n\in\Nset}\subset\calC_K$ be a maximizing sequence for \eqref{Eqn:OptProblem}. By weak compactness we can assume that $W_n$ converges weakly to $W_\infty$ and find
\begin{align*}
W_\infty\in\calC\,,\qquad
\norm{W_\infty}_2\leq \sqrt{2K}\,, \qquad \calP\at{W_\infty}= P\at{K}>0\,.
\end{align*}
Since \eqref{Lem:PotentialEnergy.E1} implies
\begin{align*}
P\at{K}\geq \calP\at{\frac{\sqrt{2K}}{\norm{W_\infty}_2}W_\infty}\geq \frac{2K}{\norm{W_\infty}_2^2}\calP\at{W_\infty}
\end{align*}
we finally conclude that $\norm{W_\infty}_2=\sqrt{2K}$, i.e. $W_\infty$ belongs  to $\calC_K$. The assertion now follows from Corollary \ref{Cor:MaximizerAreWaves}.
\end{proof}
\begin{figure}[t!]%
\centering{%
\includegraphics[width=.9\textwidth]{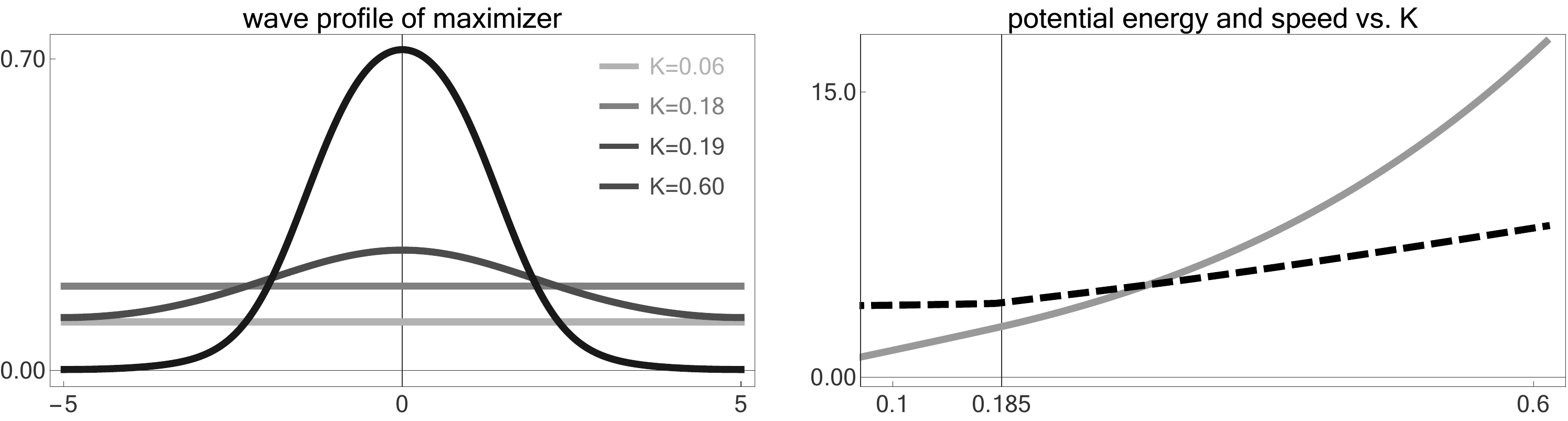}%
}%
\caption{%
\emph{Left panel:} Numerical wave profiles for the lattice \eqref{Eqn.PotThreshold} and several values of $K$, computed with the improvement dynamics from very localized initial data. The numerical maximizer is basically a constant for $K\lessapprox0.185$ but clearly localized for $K\gtrapprox0.185$.
\emph{Right panel}: Numerical values of the potential energy (gray, solid) and the speed (black, dashed) of the maximizer. %
}%
\label{fig:threshold}%
\end{figure}%
\begin{figure}[t!]%
\centering{%
\includegraphics[width=.9\textwidth]{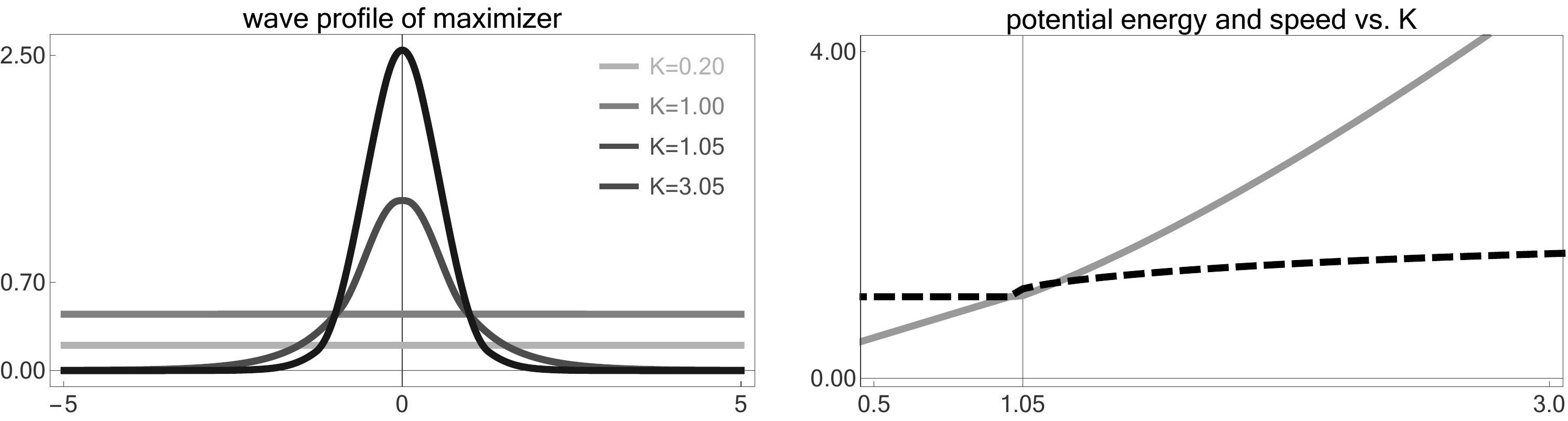}%
}%
\caption{%
Numerical simulations for the chain in \eqref{Eqn.PWLin} with localization threshold near $K=1.05$. In this example, the potential is piecewise quadratic.%
}%
\label{fig:pwlin}%
\end{figure}%
\BMHR The arguments in the proofs of Lemma \ref{Prop:ImprovementDynamics} and Lemma \ref{Lem:PeriodicMaximizers} can also be used to show that any discrete orbit of the improvement dynamics \eqref{Eqn:ImprovementDynamics} contains a strongly convergent subsequence and that any accumulation point must be a traveling wave. In particular, the estimate \eqref{Eqn:ImprovementEstimate} can be viewed as a discrete analogue to LaSalle's invariance principle. \EMHR It remains open under which conditions the optimization problem \eqref{Eqn:OptProblem} has a unique solution and whether there exists further local maxima or unstable saddle points, see also the discussion in \S\ref{sect:Intro}.
\par
We \BMHR further \EMHR emphasize that Lemma \ref{Lem:PeriodicMaximizers} does not automatically imply that the maximizer has a non-constant profile function.  For purely harmonic potentials one can show by means of Fourier transform --- see the discussion in \S\ref{sect:Aympt.Harmonic} --- that the periodic maximizer is always constant and in numerical simulations with super-quadratic functions that grow rather \BMHC  weakly \EMHC  near the origin we observe that the maximizer is constant for small $K$ but non-constant for large $K$. A typical example with discrete coupling is presented in Figure \ref{fig:threshold} and relies on
\begin{align}
\label{Eqn.PotThreshold}
M=2\,,\qquad L=5\,,\qquad \Phi_1\at{r}=\Phi_2\at{r}=\tfrac12r^2+\tfrac16r^6\,,
\end{align}
i.e. on potentials with positive second but vanishing third derivative \BMHR at $r=0$. \EMHR Similar energetic localization thresholds can be found in \cite{Her10}  for FPUT chains and in \cite{Wei99} for coherent structures in other Hamiltonian systems. Below --- see the comments to Proposition \ref{Prop.SQCriterion} and Proposition \ref{Prop:ConvergencePeriodicWaves} --- we discuss \BMHR sufficient \EMHR conditions to guarantee that the maximizer is non-constant.
\par%
Another, more degenerate and less smooth example with localization threshold is presented in Figure \ref{fig:pwlin} and concerns the chain
\begin{align}
\label{Eqn.PWLin}
M=1\,,\qquad L=4\,,\qquad \Phi\at{r}=\left\{\begin{array}{lcl}\D\tfrac12r^2&&\text{for $0\leq r\leq1$}\\\D
r-\tfrac12+\tfrac52\at{r-1}^2&&\text{for $r>1$}
\end{array}\right.
\end{align}
with piecewise linear stress strain relation $\Phi^\prime$, for which explicit approximation formulas have been derived in \cite{TV14}.
\par
We finally mention that there is no localization threshold in Hertzian chains with
monomial interaction potential due to an extra scaling symmetry. We refer to
Figure \ref{fig:hertzian} for a numerical example with
\begin{align}
\label{Eqn.PotHertzian}
M=1\,,\qquad L=4\,,\qquad \Phi\at{r}=r^p\,,
\end{align}
and to \cite{EP05,Jam12,JP14} for a variant of the corresponding improvement dynamics, the  \BMHC  super-exponential \EMHC  decay of the wave profiles, and asymptotic results for $p\gtrapprox 1$.
\begin{figure}%
\centering{%
\includegraphics[width=.9\textwidth]{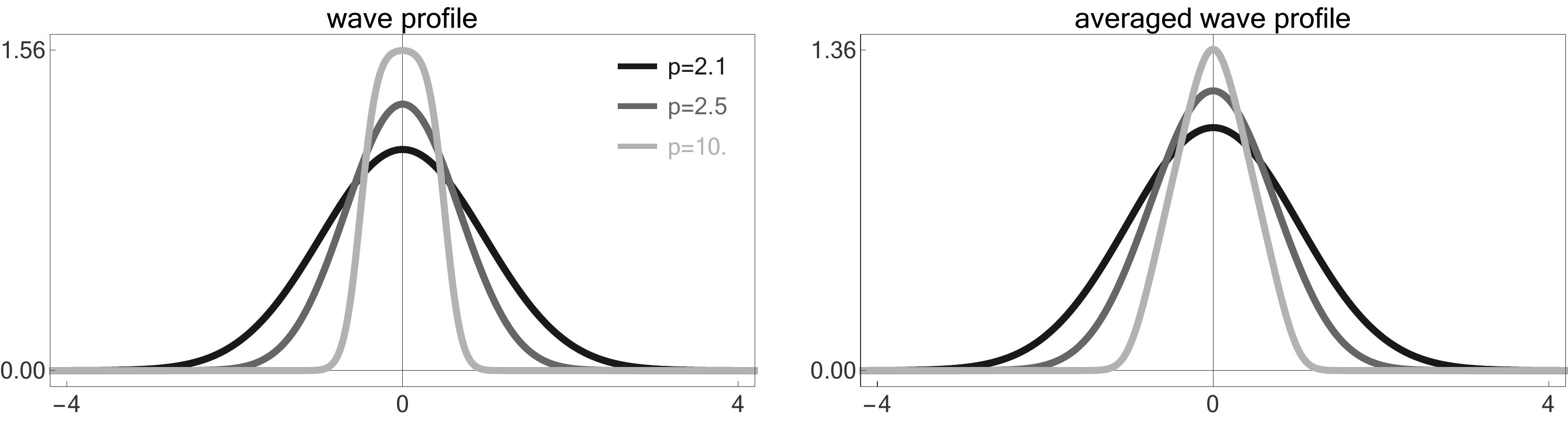}%
}%
\caption{%
\emph{Left panel:} Wave profile $W$ for the Hertzian chain \eqref{Eqn.PotHertzian} with $K=1$ and several values of the monomial exponent $p$.
\emph{Right panel}: The corresponding distance profile $\calA_1W$. %
}%
\label{fig:hertzian}%
\end{figure}%
%
%
%-------------------------------------------------------------------------------------
\subsection{Solitary waves and genuine super-quadraticity}
\label{sect:Superquadraticity} %-------------------------------------------------------------------------------------
Solitary traveling waves do not exist in the harmonic case in which the involved potentials are purely quadratic. \BMHC This follows by Fourier analysis arguments as described \EMHC \BMHR in \S\ref{sect:AsympLin}, which reveal that the linear variant of the operator $\partial\calP$  has no proper eigenfunctions but only continuous spectrum. In the variational setting, the non-existence of harmonic solitary waves becomes manifests in the \EMHR lack of strong compactness for maximizing sequences.  The usual strategy for both FPUT chains and peridynamical media is to \BMHC require \EMHC  the potentials to grow \BMHR super-quadratically \EMHR in a proper sense and to exploit concentration compactness arguments to exclude the crucial vanishing and splitting scenarios for weakly convergent sequences.  We follow a similar approach but the compactness conditions simplify since we work with unimodal functions.
\par
In what follows we call the optimization problem \eqref{Eqn:OptProblem} \emph{genuinely super-quadratic} if
\begin{align}
\label{Eqn:SQCond}
P\at{K}>Q\at{K}\,.
\end{align}
Here, $P\at{K}$ is defined in \eqref{Eqn:DefPK} and
\begin{align*}
Q\at{K}:=\sup\big\{\calQ\at{W}\;:\,W\in\calC_K\big\}
\end{align*}
with
\begin{align*}
\calQ\at{W}=\sum_{m=1}^{M}\tfrac12\Phi^{\prime\prime}_m\at{0}\int\limits_{I_L}\bat{\calA_{\xi_m} W}^2\dint{x}\qquad \text{resp.}\qquad
\calQ\at{W}=\tfrac12 \Phi^{\prime\prime}\at{0}\int\limits_0^\infty\al\at\xi\beta^2\at\xi\int\limits_{I_L}\bat{\calA_\xi W}^2\dint{x}\dint\xi
\end{align*}
represents the harmonic analogue to  \eqref{Eqn:OptProblem}. Notice that \eqref{Lem:PotentialEnergy.E1} combined with the $2$-homogeneity of $\calQ$ ensures that if \eqref{Eqn:SQCond} is satisfied for some $K=K_*$, then it also holds for  all $K\geq K_*$.
\par
We first show that \eqref{Eqn:SQCond} guarantees the existence of solitary waves and identify afterwards \BMHC sufficient \EMHC  conditions for the nonlinear potential functions. Similar ideas have been used in \cite{Her10} in the context of FPUT chains.
\begin{theorem}[strong compactness of maximizing sequences]
\label{Prop:SolitaryCompactness}
Let $L=\infty$ and suppose that $K$ satisfies \eqref{Eqn:SQCond}. Then, any sequence $\at{W_n}_{n\in\Nset}\subset\calC_K$ with $ \calP\at{W_n}\to P\at{K}$ admits a strongly convergent subsequence.
\end{theorem}
\begin{proof}
\emph{\ul{Preliminaries}}:
To elucidate the key ideas, we start with the discrete-coupling case \eqref{Case:DiscreteCoupling} and discuss the necessary modification for \eqref{Case:ContinousCoupling} afterwards. Passing to a (not relabeled) subsequence we can assume that
\begin{align*}
W_n\xrightarrow{n\to\infty}W_\infty\quad\text{weakly\;in\;$\fspaceL^2\at\Rset$}
\end{align*}
for some $W_\infty\in\calC$ with $\norm{W_\infty}_2^2\leq 2K$. Our goal is to show
\begin{align}
\label{Thm:SolitaryCompactness.PE1}
\norm{W_\infty}_2^2\geq 2K
\end{align}
because this implies in combination with the weak convergence the desired strong convergence in the Hilbert space $\fspaceL^2\at\Rset$.
\par
\emph{\ul{Truncation in $x$}}:
For given  $X>\xi_M$ we write
\begin{align*}
W_n=\widetilde{W}_n+
\ol{W}_n\qquad\text{with}\qquad \widetilde{W}_n:=\chi_{\ccinterval{-X}{+X}}W_n\,,\quad
\ol{W}_n:=\chi_{\Rset\setminus\ccinterval{-X}{+X}}W_n
\end{align*}
and observe that
\begin{align}
\label{Thm:SolitaryCompactness.PE2}
\bnorm{\widetilde{W}_n}_2^2+\bnorm{\ol{W}_n}_2^2=2K\,.
\end{align}
Moreover, \BMHR by \eqref{Eqn:EnergyFormulas1} \EMHR we have
\begin{align*}
\BMHR \calP\at{W_n}=\sum_{m=1}^M\calP_m\at{W_n}\EMHR
\end{align*}
\BMHR with \EMHR
\begin{align*}
\BMHR \calP_m\at{W_n}&:=
\int\limits_{x\in\Rset}\!\!\! \Phi_m\at{\calA_{\xi_m}W_n}\dint{x}
\\&= \EMHR
\int\limits_{\abs{x}\leq X-\xi_m/2}\!\!\! \Phi_m\at{\calA_{\xi_m}W_n}\dint{x}
\;\;+ \!\!\!\!\!\int\limits_{\abs{\abs{x}- X}\leq \xi_m/2} \!\!\!\Phi_m\at{\calA_{\xi_m}W_n}\dint{x}\;\;+\!\!\!\!\!
\int\limits_{\abs{x}\geq X+\xi_m/2} \!\!\!\Phi_m\at{\calA_{\xi_m}W_n}\dint{x} \,.
\end{align*}
The unimodality and the evenness of $\calA_{\xi_m}W_n$ combined with Lemma \ref{Lem.PropsA} and \eqref{Eqn:UnimodalDecay} imply
\begin{align*}
0\leq \int\limits_{\abs{\abs{x}- X}\leq \xi_m/2} \Phi_m\at{\calA_{\xi_m}W_n}\dint{x}\leq 2\xi_m \Phi_m\at{\frac{\xi_m\sqrt{2K}}{\sqrt{2\at{X-\xi_m/2}}}}\leq o_X\at{1}\,,
\end{align*}
where  $o_X\at{1}$ denotes a quantity that does not depend on $n$ \BMHR or $m$ \EMHR and becomes arbitrarily small as $X\to\infty$. In the same way we derive
\begin{align*}
0\leq \int\limits_{\abs{x}\leq X-\xi_m/2} \Phi_m\bat{\calA_{\xi_m}W}\dint{x}\leq  \int\limits_{\Rset} \Phi_m\bat{\calA_{\xi_m}\widetilde{W}_n}\dint{x}+o_X\at{1}=\BMHR \calP_m\bat{\widetilde{W}_n}\EMHR +o_X\at{1}
\end{align*}
as well as
\begin{align*}
0\leq \int\limits_{\abs{x}\geq X+\xi_m/2} \Phi_m\bat{\calA_{\xi_m}W}\dint{x}\leq \BMHR \calP_m\bat{\ol{W}_n}\EMHR +o_X\at{1}\leq \BMHR \calQ_m\bat{\ol{W}_n}\EMHR+o_X\at{1}
\end{align*}
due to the smoothness of $\Phi_m$ and since the convolution kernel corresponding to $\calA_{\xi_m}$ is supported in $\ccinterval{-\xi_m/2}{+\xi_m/2}$. In summary, we have shown that
\begin{align}
\label{Thm:SolitaryCompactness.PE3}
\calP\bat{W_n}\leq \calP\bat{\widetilde{W}_n}+\calQ\bat{\ol{W}_n}+o_X\at{1}
\end{align}
holds for any given $X>\xi_M$ and all $n\in\Nset$.
\par%
\emph{\ul{Further estimates}}:
By construction,  $\widetilde{W}_n$ converges weakly as $n\to\infty$ to $\widetilde{W}_\infty$ and $\calA_{\xi_m}\widetilde{W}_n$ converges (for every $m$) pointwise to $\calA_{\xi_m}\widetilde{W}_\infty$. Moreover, $\calA_{\xi_m}\widetilde{W}_n$ is pointwise bounded by $\sqrt{\xi_m}\sqrt{2K}$ and compactly supported in $\ccinterval{-X-\xi_M/2}{+X+\xi_M/2}$. We thus conclude that $A_{\xi_m}\widetilde{W}_n$ converges strongly in $\fspaceL^2\at{\Rset}$ to $\calA_{\xi_m}\widetilde{W}_\infty$ and this implies
\begin{align*}
\calP\bat{\widetilde{W}_n}\leq \calP\bat{\widetilde{W}_\infty}+o_n\at{1}\,,
\end{align*}
where $o_n\at{1}$ is allowed to depend on $X$ but vanishes in the limit $n\to\infty$. On the other hand, the super-quadraticity of $\calP$, see \eqref{Lem:PotentialEnergy.E1}, implies
\begin{align*}
\calP\bat{\widetilde{W}_\infty}\leq\frac{\bnorm{\widetilde{W}_\infty}_2^2}{2K}\calP\at{\frac{\sqrt{2K}}{\bnorm{\widetilde{W}_\infty}_2}\widetilde{W}_\infty}\leq
\frac{\bnorm{\widetilde{W}_\infty}_2^2}{2K} P\at{K}
\end{align*}
and in the same way we demonstrate
\begin{align*}
\calQ\bat{\ol{W}_n}\leq \frac{\bnorm{\ol{W}_n}_2^2}{2K}Q\at{K}\,.
\end{align*}
We thus deduce
\begin{align*}
\calP\at{W_n}\leq \frac{\bnorm{\widetilde{W}_\infty}_2^2}{2K} P\at{K}+
 \frac{\bnorm{\ol{W}_n}_2^2}{2K}Q\at{K}+o_X\at{1}+o_n\at{1}
\end{align*}
from \eqref{Thm:SolitaryCompactness.PE3}, and recalling that $\calP\at{W_n}\geq P\at{K}-o_n\at{1}$ holds  by construction, we arrive at
\begin{align}
\label{Thm:SolitaryCompactness.PE4}
P\at{K}+ \frac{\norm{\ol{W}_n}_2^2}{2K}\bat{P\at{K}-Q\at{K}}\leq \frac{\bnorm{\widetilde{W}_\infty}_2^2+\bnorm{\ol{W}_n}_2^2}{2K}P\at{K}+o_X\at{1}+o_n\at{1}
\end{align}
after writing $Q\at{K}=P\at{K}-\bat{P\at{K}-Q\at{K}}$ and rearranging terms.
\par
\emph{\ul{Justification of \eqref{Thm:SolitaryCompactness.PE1}}}:
Since the weak convergence $\widetilde{W}_n\to \widetilde{W}_\infty$ combined with \eqref{Thm:SolitaryCompactness.PE2} ensures
\begin{align*}
\frac{\bnorm{\widetilde{W}_\infty}_2^2+\bnorm{\ol{W}_n}_2^2}{2K}\leq 1+o_n\at{{1}}\,
\end{align*}
\BMHR we find the \EMHR uniform tightness estimate
\begin{align*}
\limsup_{n\to\infty}\norm{\ol{W}_n}_2^2=\frac{2K}{P\at{K}-Q\at{K}}o_X\at{1}=o_X\at{1}
\end{align*}
because letting $n\to\infty$ for fixed $X$ would otherwise produce a contradiction in \eqref{Thm:SolitaryCompactness.PE4}. Combining this with $\nnorm{\widetilde{W}_\infty}_2\leq \nnorm{W_\infty}_2$ and letting $n\to\infty$ we thus obtain
\begin{align*}
P\at{K}\leq \frac{\bnorm{W_\infty}_2^2}{2K}P\bat{K}+o_X\at{1}\,,
\end{align*}
and this yields the desired result in the limit $X\to\infty$.
\par
\emph{\ul{Concluding remarks}}:
\BMHR All arguments can be generalized to the continuous-coupling case provided that the truncation in $x$ is accompanied by an appropriate cut-off in $\xi$-space. More precisely, choosing two parameters $0<\Xi_1<\Xi_2<\infty$ we write
\begin{align*}
\calP\at{W_n}=\calP_{1}\at{W_n}+\calP_{2}\at{W_n}+\calP_{3}\at{W_n}\,,
\end{align*}
where the three contributions on the right hand side stem from splitting the $\xi$-integrating in \eqref{Eqn:EnergyFormulas2} into three integral corresponding to $\xi\in\oointerval{0}{\Xi_1}$, $\xi\in\oointerval{\Xi_1}{\Xi_2}$, and   $\xi\in\oointerval{\Xi_2}{\infty}$, respectively.
Using the identity $0\leq \Phi\at{r}\leq \tfrac12 \phi\at{r}r^2$, which follows from \eqref{Eqn:PropertiesPhi}, as well as Assumption \ref{Ass.AdmissiblePhi} and Lemma \ref{Lem.PropsA} we then estimate
\begin{align*}
0\leq \calP_1\at{W_n}&\leq
K\int\limits_{0}^{\Xi_1} \al\at\xi\beta^2\at\xi\phi\at{\xi^{1/2}\beta\at\xi \sqrt{2K}}\dint{\xi}=o_{\Xi_1}\at{1}\,,
\end{align*}
where $o_{\Xi_1}$ is independent of $n$ and vanishes for fixed $K$ in the limit $\Xi_1\to0$. Moreover, in the same way we verify
\begin{align*}
0\leq \calP_3\at{W_n}&\leq
K\int\limits_{\Xi_2}^\infty \al\at\xi\beta^2\at\xi\phi\at{\xi^{1/2}\beta\at\xi \sqrt{2K}}\dint{\xi}=o_{\Xi_2}\at{1}
\end{align*}
with $o_{\Xi_2}\at{1}\to0$ as $\Xi_2\to\infty$. To deal with the remaining
term $\calP_2\at{W_n}$, we split the $x$-integration as above and repeat all asymptotic arguments by using a combined error term $o_{X,\Xi_1,\Xi_2}\at{1}$ instead of $o_X\at{1}$. \EMHR
\end{proof}
\begin{corollary}[existence of solitary waves]
\label{Cor:ExistenceSolitaryWaves}
The constrained optimization problem \eqref{Eqn:OptProblem} admits for $L=\infty$ and any $K>0$ with \eqref{Eqn:SQCond} at least one solution $W\in\calC_K$, which provides a solitary traveling wave.
\end{corollary}
\begin{proof}
The existence of a maximizer is granted by Theorem \ref{Prop:SolitaryCompactness} and Corollary  \ref{Cor:MaximizerAreWaves} guarantees that this maximizer is a traveling wave.
\end{proof}
Notice that any solitary wave is expected to decay exponentially, where the heuristic decay rate $\la$ depends only on $\si$ and will be identified below in \eqref{Eqn:ExpDecay} using the imaginary variant of the underlying dispersion relation.  The exponential decay can also be deduced rigorously from the traveling wave equation and the properties of the convolution operator $\calA_\xi$, see \cite{HR10} for the details in an FPUT setting.
%
%
%
%-------------------------------------------------------------------------------------
\subsection{Sufficient conditions for the existence of solitary waves} %-------------------------------------------------------------------------------------
The natural strategy to verify \eqref{Eqn:SQCond} for sufficiently large $K$
is to fix a test profile $W$ and to show that the function
\begin{align*}
\lambda\mapsto \lambda^{-2}\calP\at{\la W}
\end{align*}
is unbounded as $\la\to\infty$.  A special application of this idea is the following result, which guarantees the existence of solitary waves for all $K>0$ under %rather
mild assumptions.  For simplicity we restrict our considerations to the case of continuous but finite-range  interactions and mention that similar results hold if $\xi\mapsto\xi^2 \al\at\xi\beta^2\at\xi$ decays sufficiently fast or if $\Phi_m^{\prime\prime\prime}\at{0}>0$ holds for some $m$ in the discrete-coupling case.
\begin{proposition}[criterion for genuine super-quadraticity]
\label{Prop.SQCriterion}
Condition \eqref{Eqn:SQCond} holds for $L=\infty$ in the continuous-coupling case provided that $\Phi^{\prime\prime\prime}\at{0}>0$ and $\supp\,{\al}\subset\ccinterval{0}{\Xi}$ for some $0<\Xi<\infty$.
\end{proposition}
\begin{proof}
Using Plancharel's theorem and the properties of the $\mathrm{sinc}$ function we estimate
\begin{align}
\label{Prop.SQCriterion.PE1}
\begin{split}
Q\at{K}&=\sup_{W\in\calC_K}\tfrac{1}{2}\Phi^{\prime\prime}\at{0}
\int\limits_0^{\Xi}\xi^2\al\at\xi\be^2\at{\xi}\int\limits_{\Rset}\mathrm{sinc}^2\bat{k\xi/2}\babs{\widehat{W}\at{k}}^2\dint{k}\dint\xi
\\&\leq\Phi^{\prime\prime}\at{0} K
\int\limits_0^{\Xi}\xi^2\al\at\xi\be^2\at{\xi}\dint\xi\,,
\end{split}
\end{align}
and define for any large $\Lambda$ a piecewise constant function  $W_\Lambda\in\calC$ by
\begin{align*}
W_\Lambda\at{x}=\sqrt{\frac{K}{\Lambda}}\chi_{I_\Lambda}\at{x}\,.
\end{align*}
Since the unimodal function $\calA_\xi W_\Lambda$ attains the value $\xi K^{1/2}\Lambda^{-1/2}$ for $\abs{x}\leq \Lambda-\xi/2$ and vanishes for $\abs{x}\geq \Lambda+\xi/2$, we arrive at the lower bound
\begin{align*}
\calQ\at{W_\Lambda}\geq
\tfrac{1}{2}\Phi^{\prime\prime}\at{0}
\int\limits_0^{\Xi}\al\at\xi\be^2\at{\xi}\int\limits_{-\Lambda+\xi/2}^{+\Lambda-\xi/2}
\xi^2 K \Lambda^{-1}\dint{x}\dint{\xi}
\end{align*}
and conclude with $\Lambda\to\infty$ that the upper bound in \eqref{Prop.SQCriterion.PE1} is actually sharp. Consequently, we have
\begin{align*}
\calQ\at{W_\Lambda}\geq Q\at{K}-C\Lambda^{-1}
\end{align*}
for some positive constant $C$ independent of $\Lambda$. Moreover, using the smoothness of $\Phi$ and thanks to Lemma \ref{Lem.PropsA} we estimate
\begin{align*}
\calP\at{W_\Lambda}-\calQ\at{W_\Lambda}&\geq
\at{1-o_\Lambda\at{1}} \tfrac16 \Phi^{\prime\prime\prime}\at{0} \int\limits_{0}^{\Xi}
\al\at\xi\be^3\at\xi\int\limits_{-\Lambda+\xi/2}^{+\Lambda-\xi/2}\bat{\at{\calA_\xi W_\Lambda}\at{x}}^3\dint{x}\dint\xi
\\&\geq
\at{1-o_\Lambda\at{1}} \tfrac16 \Phi^{\prime\prime\prime}\at{0} \int\limits_{0}^{\Xi}
\xi^3\al\at\xi\be^3\at\xi \frac{K^{3/2}\at{2\BMHC \Lambda \EMHC-\xi}}{\Lambda^{3/2}} \dint\xi\
\,,
\end{align*}
where the small quantity $o_\Lambda\at{1}$ is chosen such that
\begin{align*}
\sup_{0\leq\xi\leq\Xi}\;\sup_{0\leq r\leq \xi^{1/2}\be\at{\xi}\sqrt{2K}}\abs{\frac{\Phi^{\prime\prime\prime}\at{r}}{\Phi^{\prime\prime\prime}\at{0}}}\geq 1-o_\Lambda\at{1}\,.
\end{align*}
In summary, we obtain
\begin{align*}
P\at{K}\geq  \calP\at{W_\Lambda}-\calQ\at{W_\Lambda}+\calQ\at{W_\Lambda}\geq Q\at{K}-C\Lambda^{-1}+c \Lambda^{-1/2}
\end{align*}
for every large $\Lambda$ and some constant \BMHR $c>0$, \EMHR so the claim follows by choosing $\Lambda$ sufficiently large.
\end{proof}
Similar ideas reveal for large $L$ and $K$ that the constant function cannot be the maximizer for problem \eqref{Eqn:OptProblem} provided that $\Phi$ grows sufficiently fast. In fact, we have
\begin{align}
\label{Prop.SQCriterion.B1}
\calP\at{K^{1/2}\Lambda^{-1/2}}=2L  \int\limits_0^{\Xi}\Phi\at{\beta\at{\xi}\xi  K^{1/2}\Lambda^{-1/2}}\dint\xi
\end{align}
and verify for $L>\Lambda+\Xi$ the estimate
\begin{align}
\label{Prop.SQCriterion.B2}
\calP\at{W_\Lambda}&\geq \at{2M-\Xi} \int\limits_0^{\Xi}\al\at\xi\Phi\at{\beta\at{\xi}\xi  K^{1/2}\Lambda^{-1/2}}\dint\xi\,.
\end{align}
The right hands side in \eqref{Prop.SQCriterion.B2}, however, is typically larger than that from \eqref{Prop.SQCriterion.B1} provided that $r^{-2}\Phi\at{r}\to\infty$ as $r\to\infty$.
%
%
%
%-------------------------------------------------------------------------------------
\subsection{Convergence of periodic waves}
\label{sect:ConvergencePeriodicWaves}
%-------------------------------------------------------------------------------------
%
We finally show that the solitary waves provided by Corollary \ref{Cor:ExistenceSolitaryWaves} can be approximated by periodic ones. To this end we write the $L$-dependence of $P$ and $\calP$ explicitly.
\begin{proposition}[convergence of maximizers]
\label{Prop:ConvergencePeriodicWaves}
Suppose that \eqref{Eqn:SQCond} is satisfied. Then we have
\begin{align}
\label{Prop:ConvergencePeriodicWaves.E1}
P_L\at{K}\quad\xrightarrow{L\to\infty}\quad P_\infty\at{K}\,.
\end{align}
Moreover, any family $\at{W_L}_{L\geq 1}$ with $W_L$ being a maximizer of problem \eqref{Eqn:OptProblem} on $I_L$ converges along subsequences and in a strong $\fspaceL^2$-sense  to a maximizer $W_\infty$ of problem \eqref{Eqn:OptProblem} on $I_\infty=\Rset$.
\end{proposition}
\begin{proof}
\emph{\ul{Convergence of maxima}}: For finite $L$, we denote by
\begin{align*}
\widetilde{W}_L\at{x}:=W_L\at{x}\quad \text{for}\quad  \abs{x}< L\,,\qquad \widetilde{W}_L\at{x}:=0\quad\text{else}
\end{align*}
the trivial continuation of $W_L$ to a function from $\fspaceL^2\at\Rset$, which satisfies $\widetilde{W}_L\in\calC$ as well as $\calK_\infty\bat{\widetilde{W}_L}=\calK_L\at{W_L}$. Moreover, given a solitary maximizer $W_\infty$ we define for any $0<L<\infty$ the $2L$-periodic function $\ol{W}_L$ to be the scaled periodic continuation of $W_\infty$ restricted to $I_L$. This reads
\begin{align*}
\ol{W}_L\at{x}:=c_L\, W_\infty\at{ x}\quad \text{for}\quad  \abs{x}< L\,,\qquad \ol{W}_L\at{x+2L}=\ol{W}_L\at{x}\quad \text{for}\quad x\in\Rset
\end{align*}
with
\begin{align*}
c_L^{-2}:=
\at{2K}^{-1} \int\limits_{I_L}\babs{W_\infty\at{x}}^2\dint{x}
\end{align*}
and implies $\ol{W}_L\in \calC$  as well as $\calK_L\at{\ol{W}_L}=K$. Since the kernel in the convolution operator $\calA_\xi$ has compact support and since the decay estimate
\begin{align*}
\babs{W_L\at{x}}\leq \sqrt{\frac{K}{\abs{x}}}
\end{align*}
holds by \eqref{Eqn:UnimodalDecay} both for finite and infinite $L$, we readily show
\begin{align*}
\babs{\calP_\infty\bat{\widetilde{W}_L}-\calP_L\at{W_L}}\leq o_L\at{1}\,,
\end{align*}
with $o_L\at{1}\to0$ as $L\to\infty$, and deduce
\begin{align}
\label{Prop:ConvergencePeriodicWaves.P1}
P_\infty\at{K}\geq \limsup_{L\to\infty}\calP_\infty\bat{\widetilde{W}_L}\geq
\limsup_{L\to\infty}\calP_L\bat{W_L}=\limsup_{L\to\infty} P_L\at{K}\,.
\end{align}
Similarly, we find
\begin{align*}
\babs{\calP_L\bat{\ol{W}_L}-\calP_\infty\at{W_\infty}}\leq o_L\at{1}\,
\end{align*}
thanks to $\norm{\ol{W}_L-\chi_{I_L} W_\infty}_{\fspaceL^2\at{I_L}}\to0$, and this implies
\begin{align}
\label{Prop:ConvergencePeriodicWaves.P2}
\liminf_{L\to\infty} P_L\at{K} \geq \liminf_{L\to\infty} \calP_L\bat{\ol{W}_L}=\calP_\infty\at{W_\infty}=P_\infty\at{K}\,.
\end{align}
The claim \eqref{Prop:ConvergencePeriodicWaves.E1} now follows from \eqref{Prop:ConvergencePeriodicWaves.P1} and \eqref{Prop:ConvergencePeriodicWaves.P2}.
\par%
\emph{\ul{Convergence of maximizers}}:
By construction and \eqref{Prop:ConvergencePeriodicWaves.E1} we have $\calP_\infty\bat{\widetilde{W}_L}\to P\at{K}$, so \BMHR Theorem~\ref{Prop:SolitaryCompactness} \EMHC ensures the existence of strongly $\fspaceL^2$-convergent subsequences. In particular, any accumulation point $W_\infty$ belongs to $\calC$ and satisfies $\calK_\infty\at{W_\infty}=K$ \BMHR as well as \EMHR $\calP_\infty\bat{W_\infty}=P\at{K}$.
\end{proof}
\BMHR Proposition~\ref{Prop:ConvergencePeriodicWaves}
implies that the periodic solutions to the optimization problem \eqref{Eqn:OptProblem} come with non-constant profile functions provided that $L$ is sufficiently large and that condition \eqref{Eqn:SQCond} is satisfied. We further \EMHR mention that the techniques from the proof of Proposition \ref{Prop:ConvergencePeriodicWaves} can be used to characterize the continuous-coupling case as \BMHC the \EMHC scaling limit  of the discrete coupling case.
%
%
%
%-------------------------------------------------------------------------------------
\section{Outlook to asymptotic regimes}
\label{sect:AsympLin}
%-------------------------------------------------------------------------------------
%
%
%
In this section we briefly discuss some asymptotic properties of traveling waves in peridynamical media. We focus again on the continuous-coupling case \eqref{Case:ContinousCoupling} but the essential arguments can easily be modified to cover  \BMHR the lattice case \eqref{Case:DiscreteCoupling}  as well. \EMHR To ease the notation we suppose
\begin{align*}
\Phi^{\prime\prime}\at{0}=1
\end{align*}
and assume that the smooth and nonnegative coupling coefficients $\al$ and $\beta$  \BMHR sufficiently regular \EMHC for $\xi\to 0$ and $\xi\to\infty$ so that all integrals appearing below are  \BMHR well-defined. \EMHR
%
%
%-------------------------------------------------------------------------------------
\subsection{Harmonic limit for periodic waves and dispersion relation}\label{sect:Aympt.Harmonic}
%-------------------------------------------------------------------------------------
%
%
In the harmonic case $\Phi\at{r}=\tfrac12r^2$ we can solve the linear equation \eqref{Eqn:TW1} by Fourier transform. \BMHR More precisely, with \EMHR
\begin{align*}
W_{n,L,K}:=\sqrt{2KL^{-1}}\cos\at{ n L^{-1}\pi x}\,,\qquad \si^2_{n,L}:=\Theta^2\at{n L^{-1}\pi}
\end{align*}
and
\begin{align}
\label{Eqn:DefTheta}
\Theta^2\at{k}:=\int\limits_0^\infty \xi^2\alpha\at{\xi}\beta^2\at\xi \mathrm{sinc}^2\at{k\xi/2}\dint\xi\,.
\end{align}
\BMHR we obtain \EMHR the family of all even and periodic solutions, where the parameters $n\in\Nset$, $K>0$, and $0<L<\infty$ can be chosen arbitrarily. Of course, the  waves for $n>0$ are neither unimodal nor positive and not relevant for the constrained optimization problem from \S\ref{sect:Existence} \BMHR since we have \EMHR
\begin{align*}
\calP_L\bat{W_{n,K,L}}=\calQ_L\bat{W_{n,K,L}}=\Theta^2\at{nL^{-1}\pi}K
\end{align*}
but
\begin{align}
\label{Eqn:LinEnergy}
P_L\at{K}=Q_L\at{K}=K\int\limits_0^\infty\xi^2 a\at\xi\beta^2\at{\xi}\dint\xi=\Theta^2\at{0}K
\end{align}
\BMHR for the harmonic standard potential. \EMHR The dispersion relation for the peridynamical wave equation \eqref{Eqn:DynamicalSystem} is given by
\begin{align}
\label{Eqn:DefOmega}
\om^2=\Om^2\at{k}:=k^2\Theta^2\at{k}
\end{align}
and follows from inserting the ansatz $u\pair{t}{x}=\exp\at{\iu kx-\iu\om t}$ in the linear variant of the dynamical model \eqref{Eqn:DynamicalSystem2}. Elementary asymptotic arguments reveal
\begin{align*}
\Theta^2\at{k}\quad\xrightarrow{k\to0}\quad c_0:=\int\limits_0^\infty \xi^2\al\at\xi\beta^2\at\xi\dint\xi\,,\qquad
\Omega^2\at{k}\quad\xrightarrow{k\to\pm\infty}\quad c_\infty:=2\int\limits_0^\infty \al\at\xi\beta^2\at\xi\dint\xi
\end{align*}
as well as
\begin{align*}
\Theta^2\at{k}\quad\xrightarrow{k\to0}\quad\int\limits_0^\infty \xi^2\al\at\xi\beta^2\at\xi\dint\xi+O\at{k^2},
\end{align*}
provided that the coefficient functions $\al$ and $\beta$ from \eqref{Case:ContinousCoupling} are sufficiently \BMHC  non-singular \EMHC such that $\al\beta^2$ is integrable in $\cointerval{0}{\infty}$. We refer to
Figure \ref{fig:dispersion} for an illustration and mention that the choice in \eqref{Eqn.Silling} is less regular due to $\al\at\xi\beta^2\at{\xi}=\xi^{-1}$ for $\xi\approx0$.
\begin{figure}[ht!]
\centering{%
\includegraphics[width=.75\textwidth]{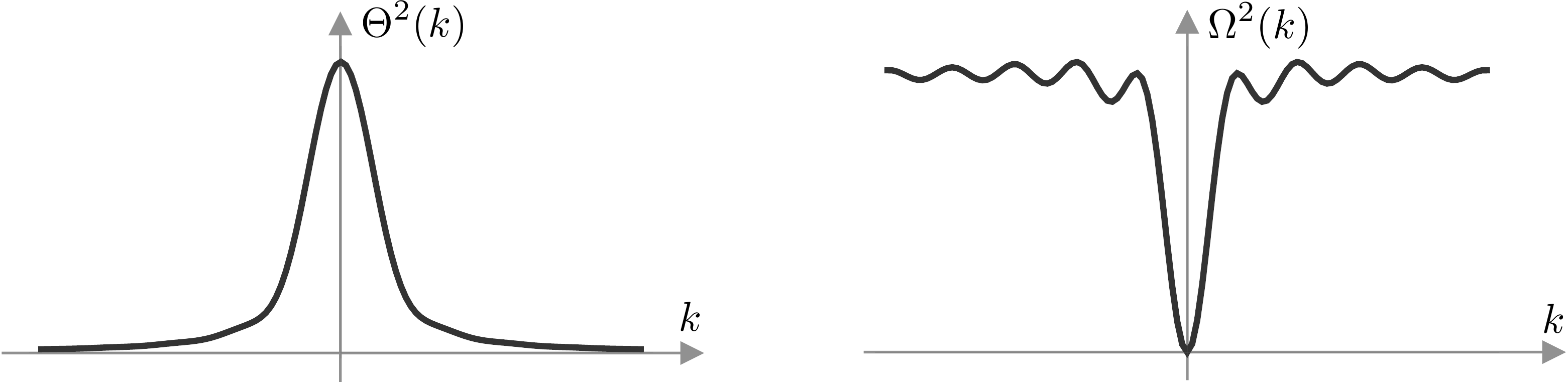}
}%
\caption{Typical example of the functions $\Theta^2$ and $\Omega^2$ from \eqref{Eqn:DefTheta} and \eqref{Eqn:DefOmega} for a peridynamical medium
with regular coefficient functions $\al$ and $\beta$. For FPUT chains,
$\Theta^2\at{k}=\mathrm{sinc}^2\at{k/2}$ has infinitely many zeros while
$\Om^2\at{k}=4\sin^2\at{k/2}$ is periodic.%
}%
\label{fig:dispersion}
\end{figure}
\par
The simplest linear PDE model that comprises the same asymptotic properties in its dispersion relation is
\begin{align*}
c_0\partial_t^2 \partial_x^2 u\pair{t}{x}-c_\infty\partial_t^2u\pair{t}{x}+c_0c_\infty \partial_x^2u\pair{t}{x}=0\,.
\end{align*}
\BMHR This equation is well-posed as \EMHR it can be written as the Banach-valued ODE
\begin{align*}
\partial_t^2u\pair{t}{x}=\calD u\pair{t}{x}\,,\qquad \calD=\at{c_0^{-1}-c_{\infty}^{-2}\partial_x^2}^{-1}\partial_x^2\,,
\end{align*}
where the bounded pseudo-differential operator $\calD$ has the symbol function
\begin{align*}
\widehat{\calD}\at{k}=\frac{c_0c_\infty k^2}{c_\infty+c_0k^2}\,.
\end{align*}
A nonlinear analogue would be \BMHR the PDE \EMHR
\begin{align}
\label{Eqn:PDESubstitute}
\partial_t^2u\pair{t}{x}=\at{c_0^{-1}-c_{\infty}^{-2}\partial_x^2}^{-1}\partial_x\Psi^\prime\bat{\partial_xu\pair{t}{x}}\,,
\end{align}
\BMHR which can be viewed as a variant of the regularized Boussinesq equation, see for instance \cite{PBSO13}. The \EMHR corresponding traveling wave ODE
\begin{align*}
-c_0W^{\prime\prime}=\si^{-2}c_0c_\infty\Psi^\prime\at{W}-c_\infty W
\end{align*}
is of Hamiltonian type and admits homoclinic solutions provided that $\Psi$ grows super-linearly and that $\si$ is sufficiently large. It \BMHC is \EMHC desirable to explore the similarities and differences between the peridynamical model \eqref{Eqn:DynamicalSystem} and the PDE substitute \eqref{Eqn:PDESubstitute} in greater detail.
\par
We finally \BMHC observe \EMHC that \BMHR the \EMHR complexified version of  $\Theta^2$ \BMHR predicts \EMHR the spatial decay of unimodal solitary waves with non-harmonic $\Phi$. In fact, the exponential ansatz
\begin{align}
\label{Eqn:ExpDecay}
W\at{x}\sim \exp\at{-\la\abs{x}}\quad\text{for}\quad \abs{x}\to\infty
\end{align}
is compatible with \eqref{Eqn:TW1} if and only if the rate parameter $\la$ satisfies the transcendental equation
\begin{align*}
\si^2=\int\limits_0^\infty \xi^2\alpha\at{\xi}\beta^2\at\xi \,\mathrm{sinch}^2\at{\xi \la/2}\dint\xi=\Theta^2\at{\iu\la}\,.
\end{align*}
This equation admits a unique positive and real solution for $\si^2>\Theta^2\at{0}$, and all solitary waves from \S\ref{sect:Existence} meet this condition due to the lower bound for the wave speed \BMHC in \EMHC Corollary \ref{Cor:MaximizerAreWaves} and since Proposition \ref{Prop:ConvergencePeriodicWaves}  combined with \eqref{Eqn:LinEnergy} implies $K^{-1}P_\infty\at{K}>K^{-1}Q_\infty\at{K}=\Theta^2\at{0}$.

%
%-------------------------------------------------------------------------------------
\subsection{Korteweg-deVries limit for solitary waves}
%-------------------------------------------------------------------------------------
%
%
%
It is \BMHC well-known \EMHC that the Korteweg-deVries (KdV) equation is naturally related to many nonlinear and spatially extended Hamiltonian systems as it governs the effective dynamics in the \BMHC large \EMHC  wave-length regime, in which traveling waves propagate with near sonic speed, have small amplitudes, and carry small energies. For solitary waves in  FPUT chains, this was first observed in \cite{ZK65} and has later been proven rigorously in \cite{FP99}. We also refer to \cite{FM14} for periodic KdV waves, to \cite{ChH18} for two-dimensional lattices, and to \cite{GMWZ14,HW17} for the more complicated case of polyatomic chains. Related results on initial value problems can be found in \cite{Sw00,HW08,HW09}.
\par
The KdV equation also governs the near sonic limit of nonlocal lattices and the corresponding existence problem for solitary lattice wave has been investigated rigorously in \cite{HML15}. In what follows we  discuss how the underlying ideas and asymptotic techniques can be applied in the peridynamical setting \eqref{Case:ContinousCoupling} provided that
\begin{align}
\notag%\label{Eqn:KdV4}
\eta:=\tfrac12\Phi^{\prime\prime\prime}\at{0}>0\,.
\end{align}
The starting point for the construction of KdV waves is the ansatz
\begin{align}
\label{Eqn:KdV5}
W\at{x}=\eps^2 \ol{W}\bat{\eps x}\,,\qquad \si^2=\Theta^2\at{0}+\eps^2
\end{align}
with sound speed $\Theta\at{0}$ from \eqref{Eqn:DefTheta}, natural scaling parameter $\eps>0$, and rescaled space variable  $\ol{x}:=\eps{x}$. Inserting this ansatz into \eqref{Eqn:TW1}, using the integral identity
\begin{align*}
\bat{\calA_\xi W}\at{x}=\eps \int\limits_{\ol{x}-\eps\xi/2}^{\ol{x}+\eps\xi/2}\ol{W}\bat{\ol{y}}\dint\ol{y}=\bat{\eps\calA_{\eps\xi}\ol{W}}\at{\ol{x}}\,,
\end{align*}
and dividing by $\eps^4$ we obtain a transformed integral equation for $\ol{W}$. This equations can be written as
\begin{align}
\label{Eqn:KdV1}
\calB_\eps V = \calN_\eps\at{V}+\eps^2 \calG_\eps\at{V}\,,
\end{align}
where the  operators
\begin{align*}
\calB_\eps \ol{W}:=\ol{W}+\eps^{-2}\Theta^2\at{0}\ol{W}-\eps^{-4}\int\limits_0^{\infty}\al\at\xi\beta^2\at\xi\calA_{\eps\xi}^2\ol{W}\dint\xi
\end{align*}
and
\begin{align*}
\calN_\eps\at{\ol{W}}:=\eps^{-5}\eta\int\limits_0^{\infty}\al\at\xi\beta^3\at\xi\calA_{\eps\xi}\bat{\eps\calA_{\eps\xi}\ol{W}}^2\dint\xi
\end{align*}
collect all terms that are linear and quadratic, respectively, with respect to $\ol{W}$, while
\begin{align*}
\calG_\eps\at{\ol{W}}:=\eps^{-7}\int\limits_0^{\infty}\al\at\xi\beta\at\xi\calA_{\eps\xi}G\bat{\eps\beta\at\xi\calA_{\eps\xi}\ol{W}}\dint\xi
\end{align*}
with $G\at{r}:=\Phi^\prime\at{r}-r-\eta r^2=O\at{r^3}$ stems from the cubic and the higher order terms and does not contribute to the KdV limit.
\par%
The key asymptotic observations for the limit $\eps\to0$ is the formal asymptotic expansion
\begin{align}
\label{Eqn:KdV2}
\calA_\delta U =\delta U+\tfrac{1}{24}\delta^3U^{\prime\prime}+\tdots\,,
\end{align}
which provides in combination with the above formulas the formal limit equation
\begin{align*}
\ol{W}-c_1\ol{W}^{\prime\prime}=c_2 \ol{W}^2
\end{align*}
with positive coefficients
\begin{align}
\label{Eqn:KdV7}
c_1 := \tfrac12\int\limits_0^\infty\xi^4\al\at\xi\beta^2\at\xi\dint\xi>0\,,\qquad c_2:= \eta\int\limits_0^\infty\xi^3\al\at\xi\beta^3\at\xi\dint\xi>0\,.
\end{align}
This planar Hamiltonian ODE is precisely the traveling wave equation of the KdV equation and admits the homoclinic and even solution
\begin{align}
\label{Eqn:KdV8}
\ol{W}_{\text{KdV}}\at{\ol{x}}:=\frac{3}{2c_2}\mathrm{sech}^2\at{\frac{\ol{x}}{2\sqrt{c_1}}}\,,
\end{align}
which is moreover unique, unimodal, positive, and exponentially decaying. However, the  asymptotic expansion in \eqref{Eqn:KdV2} is not regular but singular because the error terms involve higher derivatives. The rigorous justification of the limit problem \eqref{Eqn:KdV7} is therefore not trivial and necessitates the use of refined asymptotic techniques. Following \cite{FP99}, we use the corrector ansatz
\begin{align}
\label{Eqn:KdV6}
\ol{W}=\ol{W}_{\text{KdV}}+\eps^2\ol{V}\,,
\end{align}
which transforms \eqref{Eqn:KdV1} into
\begin{align*}
\calB_\eps \ol{V}-\calM_\eps \ol{V}:=\calF_\eps\at{\ol{V}}\,,
\end{align*}
where the linear operator
\begin{align*}
\calM_\eps \ol{W}=
\eps^{-5}2\eta\int\limits_0^{\infty}\al\at\xi\beta^3\at\xi\calA_{\eps\xi}\Bat{\bat{\eps\calA_{\eps\xi}\ol{W}_{\text{KdV}}}\bat{\eps\calA_{\eps\xi}\ol{V}}}\dint\xi
\end{align*}
represents the additional linear term that stems from the linearization of $\calN_\eps$ around $\ol{W}_{\text{KdV}}$. Moreover, the nonlinear operator
\begin{align*}
\calF_\eps\at{\ol{V}}=\eps^{-2}\Bat{\calB_\eps \ol{W}_{\text{KdV}}-\calQ_\eps\at{\ol{W}_{\text{KdV}}}}+\eps^2\calQ_\eps\at{\ol{V}}+\calG_\eps\bat{\ol{W}_{\text{KdV}}+\eps^2\ol{V}}
\end{align*}
is composed of small nonlinear terms in $\ol{V}$ as well as residual terms depending on $\eps$ and $\ol{W}_{\text{KdV}}$ only.
\par
\begin{figure}[ht!]
\centering{%
\includegraphics[width=.75\textwidth]{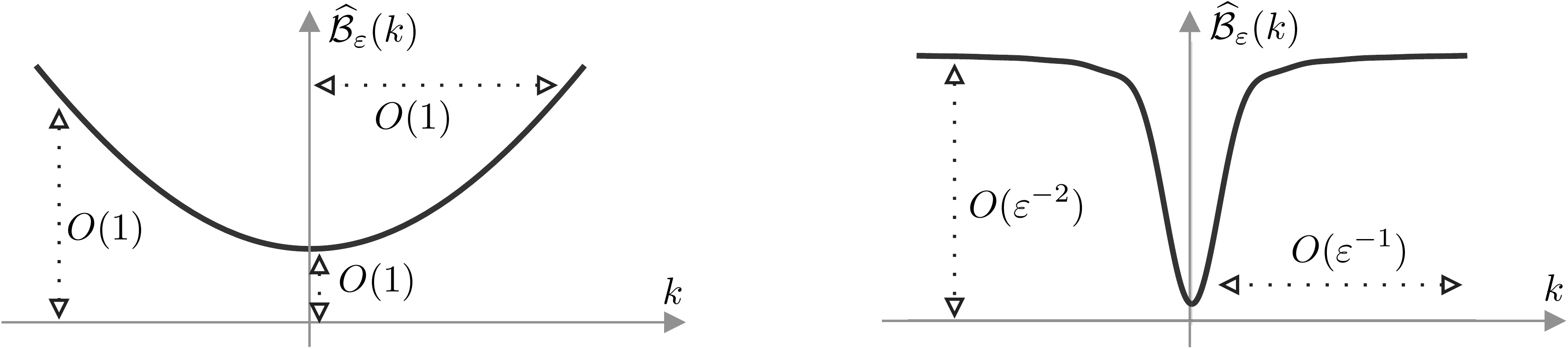}
}%
\caption{Cartoon of the symbol function from \eqref{Eqn:KdVS}, plotted on a normal  (left) and a large (right) scale.%
}%
\label{fig:kdv}
\end{figure}
The rigorous justification of the KdV limit can \BMHR be \EMHR achieved along the lines of \cite{HML15} and hinges on the following crucial ingredients:
\begin{enumerate}
\item
The pseudo-differential operator $\calB_\eps$ is uniformly invertible and its inverse is almost compact in the sense that it can be written as the sum of a compact operator and a small bounded one. These statements follow from an asymptotic analysis of the corresponding symbol function
\begin{align}
\label{Eqn:KdVS}
\widehat{\calB}_\eps \at{k}:=1+\int\limits_0^{\infty}\xi^2\al\at\xi\beta^2\at\xi\bat{1-\eps^{-2}\mathrm{sinc}^2\bat{\eps \xi k/2}}\dint\xi\,,
\end{align}
see Figure \ref{fig:kdv} for an illustration.
\item
The symmetric operator $\calL_\eps:=\calB_\eps-\calM_\eps$ is uniformly invertible on  $\fspaceL^2_{\text{even}}\at\Rset$ as it
satisfies
\begin{align*}
\norm{\calL_\eps \ol{V}}_2\geq c \norm{\ol{V}}_2
\end{align*}
on that space, where the constant $c>0$ can be chosen \BMHC independently \EMHC of $\eps$. This can be proven by contradiction using the almost compactness of $\calB_\eps^{-1}$ and the fact, that the limiting operator
\begin{align*}
\calL_0\ol{V}=\ol{V}-c_1\ol{V}^{\prime\prime}-2c_2 \ol{W}_{\text{KdV}}\ol{V}
\end{align*}
has a one-dimensional nullspace in $\fspaceL^2\at\Rset$ according to the Sturm-Liouville theory, which is spanned by the odd derivative of the KdV profile \eqref{Eqn:KdV8}.
\item
\BMHR For sufficiently small $\eps>0$, the operator $\calL_\eps^{-1}\calF_\eps$ maps  a certain ball in $\fspaceL^2_\even\at\Rset$ contractively into itself so that  there exists a locally unique fixed point. \EMHC
\end{enumerate}
In particular, for any sufficiently small $\eps>0$ there exists via \eqref{Eqn:KdV5} and \eqref{Eqn:KdV6} a near-sonic wave solution to \eqref{Eqn:TW1}. The nonlinear stability for such waves has been shown in the paper series  \cite{FP02,FP04a, FP04b} for FPUT chains, see also \cite{HW13b}. It remains a challenging task to generalize the proofs to the peridynamical case.
%
%
%
%-------------------------------------------------------------------------------------
\subsection{High-speed limit for super-polynomial potentials}
%-------------------------------------------------------------------------------------
%
%
%
Another well-known asymptotic regime concerns atomic chains with super-polynomially growing interaction potential and waves which propagate much faster than the sound speed and carry a huge amount of energy. It has been observed in \cite{FM02, Tre04} for FPUT chains with singular Lennard-Jones-types potential that both the distance and the velocity profiles of those high-speed waves converge to simple limit functions which are naturally related to the traveling waves in the hard-sphere model. In other words, the particles interact asymptotically by \BMHC elastic \EMHC collisions and \BMHR traveling waves come with strongly localized profile functions. \EMHC The asymptotic analysis of the high-energy limit of FPUT chains and the underlying advance-delay-differential equation has later been refined by the authors. In \cite{HM15,Her17} they derive accurate and almost explicit approximation formulas for the wave profiles by combining a nonlinear shape ODE with local scaling arguments and natural matching conditions. Moreover, the local uniqueness, smooth parameter dependence, and nonlinear orbital stability of high-speed waves are established in \cite{HM17,HM18} using similar two-scale \BMHR techniques, the \EMHR non-asymptotic part of the Friesecke-Pego theory from \BMHR \cite{FP02,FP04a}, and the enhancement by Mizumachi in \cite{Miz09}. \EMHR
\begin{figure}[t!]%
\centering{%
\includegraphics[width=.9\textwidth]{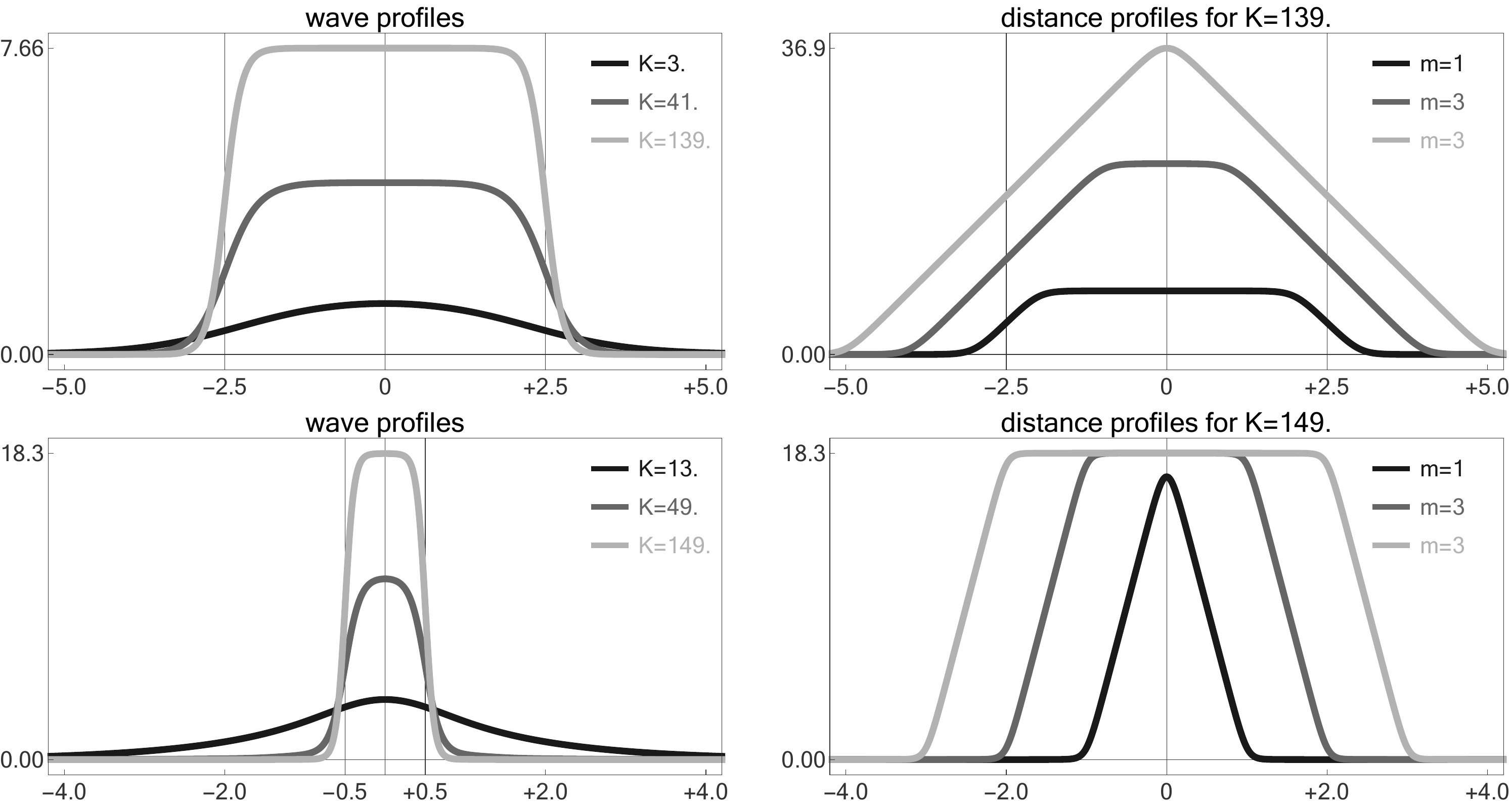}%
}%
\caption{%
\emph{Top row:} Wave profile $W$ for example \eqref{ExLimit} with
$\beta_m=1$ and several values of $K$ (\emph{left}) as well as the corresponding distance profiles $\calA_{m}W$ (right).
\emph{Bottom row:} Similar plots with $\beta_m=m^{-1}$, which is the typical choice for peridynamics.%
}%
\label{fig:limit}%
\end{figure}%
\par
Preliminary computations as well as numerical simulations indicate that the high-energy
waves in peridynamical media can likewise be approximated by simple profile functions.
For two examples with discrete coupling we refer to Figure \ref{fig:limit}, which relies
on
\begin{align}
\label{ExLimit}
M=5\,,\qquad L=10\,,\qquad \Phi_m\at{r}=\cosh\at{\beta_m r}\,,
\end{align}
and illustrates that the wave profile $W$ converges as $K\to\infty$ to the scaled indicator function \BMHR of an interval \EMHR provided that the interaction potentials grow faster than a polynomial. This implies that the corresponding distance profiles are basically piecewise affine but the details depend on the choice of the parameters.
\par
A rigorous asymptotic analysis of this numerical observation lies beyond the scope of this paper but first steps are already made in \cite{Che17} for the nonlocal advance-delay-differential equations that govern solitary waves in two-dimensional FPUT lattices. More  precisely, assuming \BMHC exponentially \EMHC growing potentials and using the multi-scale techniques from \cite{Her17} one can derive a reduced Hamiltonian ODE system that can be expected to describe the local rescaling of the different distance profiles. This ODE, however, has more than two degrees of freedom and a qualitative or quantitative analysis of the relevant solutions is hence not trivial.  It remains a challenging task to characterize the high-energy limit for peridynamical media or chains with more than nearest-neighbor interactions, and to justify the numerical observations \BMHC in \EMHC Figure \ref{fig:limit}.
%
%
%
%
% ------------------------------------------------------------------------------------
\section*{Acknowledgements}
Some aspects of our work build upon analytical or numerical investigations in the theses of Benedikt Hewer and Fanzhi Chen, see \cite{Hew13,Che13}. We also acknowledge the financial support by the \emph{Deutsche Forschungsgemeinschaft} (DFG individual grant HE 6853/2-1).
%
%\bibliographystyle{alpha}
%\bibliography{paper}
%

%
%
\end{document}